\documentclass[11pt]{article}
\usepackage{amsfonts}
\usepackage{amssymb}
\usepackage{graphics}
\usepackage{amsmath,epsfig,psfrag}
\usepackage[english]{babel}
\usepackage[all]{xy}
\usepackage{amscd}
\usepackage{amsthm}
\usepackage{latexsym}
\usepackage{graphicx}
\usepackage{mathrsfs}
\usepackage{multirow}
\usepackage{listings}
\usepackage{courier}
\usepackage{chngpage}

\def\Z{{\Bbb Z}}

\newtheorem{theorem}{Theorem}[section]

\newtheorem{proposition}[theorem]{Proposition}
\newtheorem{lemma}[theorem]{Lemma}

\theoremstyle{definition}
\newtheorem{example}[theorem]{Example}
\newtheorem{conjecture}[theorem]{Conjecture}

\theoremstyle{remark}
\newtheorem{remark}[theorem]{Remark}
\begin{document}
\title{Graphs in the 3-sphere with maximum symmetry}
\author{Chao Wang, Shicheng Wang, Yimu Zhang and Bruno Zimmermann}

\date{}
\maketitle

\begin{abstract}
We consider the orientation-preserving actions of finite groups $G$ on pairs $(S^3, \Gamma)$,
where $\Gamma$ is a connected graph of genus $g>1$, embedded
in $S^3$. For each $g$ we give the maximum order $m_g$ of such $G$ acting on
$(S^3, \Gamma)$ for all such $\Gamma\subset S^3$.
Indeed we will classify all graphs $\Gamma\subset S^3$ which realize these $m_g$ in different levels: as abstract graphs and as spatial
graphs, as well as their group actions.

 Such maximum orders without the condition ``orientation-preserving'' are also addressed.
\end{abstract}

\begin{tabular}{@{}l@{ }p{10.1cm}} {\bf Keywords }
symmetry of graph, symmetry of 3-sphere, extendable action
\end{tabular}


\tableofcontents

\section{Introduction}
We will study smooth, faithful actions of finite groups $G$ on pairs
$(S^3, \Gamma)$, where $\Gamma$ denotes a finite connected graph
with an embedding $e: \Gamma\to S^3$. We also say that such a $G$-action
on $\Gamma$ is \textit{extendable} (w.r.t. $e$). Let $g(\Gamma)$ denote the
genus of $\Gamma$, defined as the rank of its
fundamental group (which is a free group).
We will always assume that $g>1$ in the present paper.

Except in Section 5 we will only consider orientation-preserving
finite group actions on $S^3$. Referring to the recent
geometrization of finite group actions on $S^3$ (see \cite{Pe}), we will consider
only orthogonal actions of finite groups on $S^3$, i.e., finite
subgroups $G$ of the orthogonal group $SO(4)$.

Let $m_g$ denote the maximum order of such a group $G$ acting on a
pair $(S^3, \Gamma)$, for all embeddings of finite graphs $\Gamma$
of a fixed  genus $g$ into $S^3$. In this paper we will determine
$m_g$ and classify all finite graphs $\Gamma$ which realize the
maximum order $m_g$.

A similar question for the pair $(S^3, \Sigma_g)$, where $\Sigma_g$
is the closed orientable surface of genus $g$, was studied in
\cite{WWZZ}. The corresponding maximum order $OE_g$ of finite
groups acting on $(S^3, \Sigma_g)$ for all possible embeddings
$\Sigma_g\subset S^3$ was obtained in that paper.

Let $V_g$ denote the handlebody of genus $g$. Each graph
$\Gamma\subset S^3$ of genus $g$ has a regular neighborhood which is homeomorphic to $V_g$.
Call $\Gamma$ {\it unknotted} if the complement of its regular
neighborhood is also a handlebody, and otherwise {\it knotted}.

Our first result is from a
simple but significant observation, which claims $m_g=OE_g$,  and can be considered as a version of [WWZZ, Theorem 1.1].

\begin{theorem}\label{mg}
For each  $g>1$, $m_g=OE_g$, therefore  $m_g$ are given in the
following table.
\begin{center}
\begin{tabular}{|c|c|}
\hline $m_g$ & $g$\\\hline $12(g-1)$ & $2, 3, 4, 5, 6, 9, 11, 17
,25, 97, 121, 241, 601$\\\hline $8(g-1)$ & $7, 49, 73$\\\hline
$20(g-1)/3$ & $16, 19, 361$\\\hline $6(g-1)$ & $21, 481$\\\hline
$24(g-1)/5$ & $41$\\\hline $30(g-1)/7$ & $29, 841, 1681$\\\hline
$4(\sqrt{g}+1)^2$ & $g=k^2, k\neq3, 5, 7, 11, 19, 41$\\\hline
$4(g+1)$ & remaining numbers\\\hline
\end{tabular}
\end{center}
Moreover $m_g$ is realized by unknotted graphs for all $g$ except
$g=21, 481$.
\end{theorem}

\medskip

The difference between graph case and surface case is that, given a
fixed $g$, there is a unique closed orientable surface of genus $g$,
but there are infinitely many graphs of genus $g$. We will consider
only {\it minimal} graphs, i.e, the graphs without free edges and without
extra vertices: an edge is called {\it free} if one of its vertices has
degree one; a vertex is {\it extra} if it has degree two and its
stable subgroup (stablizer in $G$) equals the stable subgroup of each adjacent
edge. Clearly there are only finitely many minimal graphs of genus
$g$ for each $g>1$. Note that this is really no restriction since we
can delete all free edges and extra vertices (or, vice versa, add
arbitrarily many free edges  and extra vertices in a $G$-equivariant
way) without changing the genus, and since the group actions we
considered are in $SO(4)$ (see Proposition \ref{SO(4)}).

The major part of the paper is to classify all minimal graphs
$\Gamma\subset S^3$ which realize these maximum orders $m_g$. To be precise and brief, we need some definitions.
Suppose $G$ acts on $(S^3, \Gamma)$ for a minimal graph $\Gamma$
embedded in $S^3$ such that $|G|=m_g$. We call $\Gamma$ an
\textit{(abstract) MS graph} (of genus $g$), $(S^3, \Gamma)$ a \textit{spatial (MS)
graph} for $\Gamma$, and we call the $G$-action on $(S^3, \Gamma)$
resp. its restriction to $\Gamma$ an \textit{MS action} for $(S^3, \Gamma)$ resp. $\Gamma$, where MS represents ``maximum symmetry".
In the above definitions, we may ignore the words ``abstract", ``of genus $g$" and  ``MS" in the parentheses
when there is no confusion.

We say that two MS actions on abstract MS graphs are {\it equivalent} if they are conjugated by an isomorphism between the two graphs. We say that two spatial MS graphs are {\it equivalent} if there is a diffeomorphism of $S^3$ sending one to the other, and we say that two MS actions on spatial MS graphs are {\it equivalent} if they are conjugated by a diffeomorphism of $S^3$. Note that these diffeomorphisms can be orientation-reversing.

It is natural to ask: For given $g$, (i) what are the MS graphs?
(ii) For a fixed MS graph $\Gamma$, what are the MS actions on
$\Gamma$? (iii) What are the spatial MS graphs $(S^3, \Gamma)$ for
$\Gamma$? (iv) For a fixed spatial MS graph $(S^3, \Gamma)$,
what are the MS actions for $(S^3, \Gamma)$? We have the following
theorem.


\begin{theorem}\label{MS-graph}

(1) For each $g$, the number of abstract MS graphs of genus $g$ is four
for $g=11, 241$; is two for $g=3,5,7,17, 19, 29, 41, 97, 601, 841,
1681$; and is one for all the remaining $g$.

(2) For each abstract MS graph $\Gamma$ of genus $g$, the number of spatial
MS graphs for $\Gamma$ is infinite for $\Gamma$ of genus $21,
481$; is two for $\Gamma$ of genus $9, 121, 361$ and for one $\Gamma$ of
genus $11$; and is one for all the remaining MS graphs.


(3)  Each spatial MS graph has a unique MS action. This is also true for each (abstract) MS graph except for one graph of genus 29.

\end{theorem}

Actually Theorem \ref{MS-graph} will be included in very precise
results (Theorem \ref{detailed-classification} and also Appendix A) in the coming text, which present all the  MS
graphs, as well as spatial MS graphs except $g=21, 481$. Of
course, the meaning of ``present all abstract and spatial MS graphs"
itself will be addressed soon.

Our approach relies on \cite{WWZZ} which builds the connection
between the study of $OE_g$ and orbifold theory. The paper is
organized as below.


In Section \ref{orbifold}, we will give a brief introduction to the
orbifold theory, and introduce necessary terminologies to present
Theorem \ref{classify}, the main result of \cite{WWZZ}, which is a
list of spherical orbifolds $\mathcal{O}$ with marked (allowable)
singular edges and dashed arcs. Indeed we also try to outline the
ideas of \cite{WWZZ}.

In Section 3 by some quick arguments based on Section 2 we first
prove Theorem \ref{mg}. Then  by  picking information from Theorem
\ref{classify} exactly related to the maximum order $m_g$, and refining this information with respect
to the graph case, we present Theorem \ref{position} which list all
spatial MS graphs in the following sense: $\Gamma$ is a spatial MS
graph if and only if $\Gamma=p^{-1}(a)$, where $a$ is a marked
singular edge or dashed arc of a 3-orbifold $\mathcal{O}$ in Theorem
\ref{position} and $p:S^3\to \mathcal{O}$ is the orbifold covering.

Note that the information provided by Theorem \ref{position} in terms of
orbifolds does not tell us the following: Suppose $p^{-1}(a)$ and
$p^{-1}(b)$ are two MS graphs of genus $g$ provided by Theorem
\ref{position}. (1) Are they the same abstract graph? and if yes,
are they the same spatial graph? And more naively: (2) Can we see
$\Gamma=p^{-1}(a)$ as an abstract MS graph and as a MS spatial
graph  intuitively? And, as any graph theorist would ask, what are the primary graph invariants of those graphs?

In Section 4, various methods are introduced to give the detailed
classification result Theorem \ref{detailed-classification}, which
give a precise answer to Question (1).

The answer to Question (2) for abstract graphs is in Section \ref{pictureG} and Appendix
\ref{appendixA}, which gives a table of all MS graphs with various
invariants.
Appendix \ref{appendixB} is devoted to answer Question
(2) for spatial graphs, where we try to visualize those $\Gamma\subset S^3$ by stereographically projecting them onto $R^3$, at least for general cases and the cases related to classical regular polyhedra.
The pictures in both Appendices \ref{appendixA} and \ref{appendixB} are
produced by the computer with assistance of \cite{GAP} and \cite{Mathematica}.


Certainly the roles of the figures in those appendices are limited,
since they are hard to see for large $g$, especially for spatial graphs.
An alternative  intuition  of those symmetries
are given in the  section  ``Intuitive view of large symmetries of $(S^3, \Sigma_g)$"  in \cite{WWZZ1} via spherical tessellations and equivariant Dehn surgeries, also see \cite{Du2}, \cite{WWZZ} and \cite{Wa}.

In Section 5 we will discuss maximum orders of extendable finite
group actions on $(S^3, \Gamma)$ allowing  orientation-reversing
elements based on the results in previous sections and \cite{WWZ} . We will see  differences between the
maximum orders of arbitrary  graphs and of minimal graphs: The
former can be determined and the later are still unknown for some values of $g$. To be precise, let  $M_g$ be the general maximum orders of
extendable group actions on minimal graphs  of genus $g$. Then $M_g$
are given in the following table.

\begin{center}
\begin{tabular}{|c|c|}
\hline $M_g$ & $g$\\\hline $24(g-1)$ & $3, 4, 5, 6, 11, 17,
97, 601$\\\hline $16(g-1)$ & $7, 9, 73$\\\hline $40(g-1)/3$ & $16,
19$\\\hline $12(g-1)$ & $2, 25, 121, 241$\\\hline $48(g-1)/5$ &
$41$\\\hline $60(g-1)/7$ & $29, 841, 1681$\\\hline $8(\sqrt{g}+1)^2$
& $k^2$, $k\neq 11$\\\hline $8(g+1)> M_g \ge 4(g+1)$ & remaining numbers\\\hline
\end{tabular}
\end{center}

\begin{conjecture}
Suppose $g$ is neither a square number  nor one of those finitely
many $g$ listed in the table above. Then $M_g$ is $4(g+1)$ for prime
$g$, and $4(p+1)(q+1)$ otherwise, where $pq=g$, $p$ is the smallest
non trivial divisor of $g$.
\end{conjecture}

\bigskip\noindent\textbf{Acknowledgement}.
We thank the referee for his many valuable comments which considerably improved our paper.

 The first three authors are partially supported by grants No.11501534, No.11371034 and No.11501239 of the National Natural Science Foundation of China respectively.

\section{3-orbifold and main results in \cite{WWZZ}}\label{orbifold}

For orbifold theory, see \cite{Th}, \cite{Du1} or \cite{BMP}. We
give a brief introduction here for later use.

All of the $n$-orbifolds that we considered have the form $M/H$. Here $M$ is an
orientable $n$-manifold and $H$ is a finite group acting faithfully
on $M$, preserving orientation. For each point $x\in M$, denote
its stable subgroup by $St(x)$, its image in $M/H$ by $x'$. If
$|St(x)|>1$, $x'$ is called a \textit{singular} point with \textit{index} $|St(x)|$,
otherwise it is called a \textit{regular} point. If we forget the singular
set we get the topological \textit{underlying space} $|M/H|$ of the orbifold.

We can also define covering spaces and the fundamental group of an
orbifold. There is an one-to-one correspondence between orbifold
covering spaces and conjugacy classes of subgroups of the
fundamental group, and regular covering spaces correspond to normal
subgroups. A Van-Kampen theorem is also valid, see \cite[Corollary 2.3]{BMP}. In
the following, automorphisms, covering spaces and fundamental groups
always refer to the orbifold setting.

We call $B^n/H$ (resp. $S^n/H,V_g/H)$ the \textit{discal (resp. spherical, handlebody)
orbifold}. Here $B^n$ (resp. $S^n$) denotes the $n$-dimensional ball (resp. sphere).
By classical results, $B^2/H$ is a disk, possibly with one singular
point; $B^3/H$ belongs to one of the five models in Figure
\ref{fig:1}, corresponding to the five classes of finite subgroups
of $SO(3)$. Here the labeled numbers denote indices of interior
points of the corresponding edges. $V_g/H$ can be obtained by
pasting finitely many $B^3/H$ along some $B^2/H$ in their
boundaries. It is easy to see that the singular set of a 3-orbifold $M/H$
is always a trivalent graph $\Theta$.

\begin{figure}[h]
\centerline{\scalebox{0.6}{\includegraphics{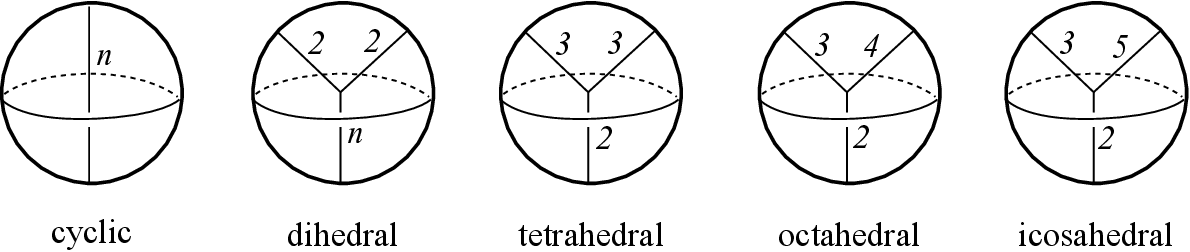}}} \caption{Five
models}\label{fig:1}
\end{figure}

Suppose $G$ acts on $(S^3, \Sigma_g)$. Call a 2-orbifold
$\mathcal{F}=\Sigma_g/G \subset \mathcal{O}=S^3/G$ {\it
allowable} if $|G|>4(g-1)$. A sequence of observations about
allowable 2-orbifolds were made in \cite{WWZZ} (Lemma 2.4, 2.7, 2.8,
2.9, 2.10), in particular: Suppose $\mathcal{F} \subset \mathcal{O}$
is allowable, then (i) $|\mathcal{O}|=S^3$; (ii) $\mathcal{F}\subset
\mathcal{O}$ is $\pi_1$-surjective; (iii) $|\mathcal{F}|=S^2$ with
four singular points having one of the following types:
$(2,2,2,n)(n\ge 3)$, $(2,2,3,3)$, $(2,2,3,4)$, $(2,2,3,5)$; and very
crucially (iv) $\mathcal{F}$ bounds a handlebody orbifold, which is
a regular neighborhood of either an edge of the singular set or a
dashed arc, presented in (a) or (b) of Figure \ref{fig:2}. Here
labels are omitted in (a), and more description of (b) will be given
later. (i) allows us to consider only Dunbar's famous list in
\cite{Du1} of all spherical 3-orbifolds with underlying space $S^3$.
Searching for all possible 2-suborbifolds that satisfy the
conditions (ii), (iii) and (iv) by further analysis from
topological, combinatoric, numerical, and group theoretical aspects
leads to a list in Theorem 6.1 of \cite{WWZZ}, presented here as
Theorem \ref{classify}. We will first need
to explain the terminology in the statement of Theorem \ref{classify} and the notation in the accompanying tables.

\begin{figure}[h]
\centerline{\scalebox{0.5}{\includegraphics{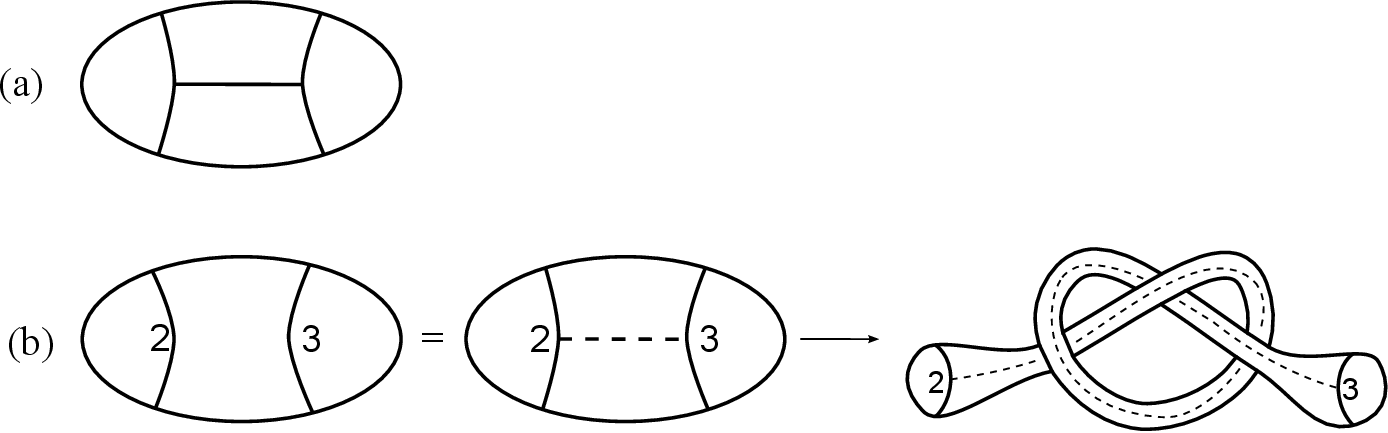}}}
\caption{Handlebody orbifolds}\label{fig:2}
\end{figure}

Since all the spherical 3-orbifolds we considered have underlying
space $S^3$, they are determined by their labeled singular trivalent
graphs. From now on, a singular edge always means an edge of
$\Theta$, the singular set of the orbifold;
singular edges with index $2$ are not labeled; and a dashed arc is
always a regular arc with two ends at two edges of $\Theta$ with
indices 2 and 3 as in Figure \ref{fig:2}(b). An
edge/dashed arc is {\it allowable} if the boundary of its regular
neighborhood is an allowable 2-orbifold.

For each 3-orbifold $\mathcal{O}$ in the list, the order of
$\pi_1(\mathcal{O})$ is given first. Then singular edges/dashed
arcs are listed, which are marked by letters $a, b, c,...$ to denote
the boundaries of their regular neighborhoods. Then singular types
of the boundaries and genera of their pre-images in $S^3$ are given.
When the singular type is $(2, 2, 3, 3)$, there are two subtypes
denoted by I and II, corresponding to Figure \ref{fig:2}(a) and
Figure \ref{fig:2}(b) (exactly the dashed arc case).

We say that an orientable separating 2-suborbifold (2-subsurface)
$\mathcal{F}$ in an orientable 3-orbifold (3-manifold) $\mathcal{O}$
is \textit{unknotted} or \textit{knotted}, depending on whether it bounds handlebody
orbifolds (handlebodies) on both sides. A singular edge/dashed arc
is \textit{unknotted} or \textit{knotted}, depending on whether the boundary of its
regular neighborhood is unknotted or knotted.

If a marked singular edge/dashed arc is knotted, then it has a subscript `$k$'. If a marked dashed arc is unknotted, then there also
exists a knotted one (indeed infinitely many) and it has a subscript `$uk$'. Call two singular edges/dashed arcs \textit{equivalent},
if there is an orbifold automorphism sending one to the other, or
the boundaries of their regular neighborhoods as 2-orbifolds are
orbifold-isotopic.

The way to list orbifolds in Theorem \ref{classify} is influenced by the lists of \cite{Du1} and \cite{Du2}. The labels below the orbifold pictures come from \cite[Tables I, II and III]{WWZZ}. In picture 15E and picture 19, the letter $n$ refers to particular choices of parameters in infinite families.

\begin{theorem}\label{classify}
Up to equivalence, the following tables list all allowable singular
edges/dashed arcs except those of type II. In the type II case, if
there exists an allowable dashed arc in some $\mathcal{O}$, then
$\mathcal{O}$ and one such arc in it are listed. The arc will be
unknotted if there exists an unknotted one in $O$.
\end{theorem}
\begin{table}[h]
\caption{Fibred case: type is $(2, 2, 3, 3)$}\label{tab:fibred2233}
\end{table}
\begin{center}
\scalebox{0.4}{\includegraphics*[0pt,0pt][152pt,152pt]{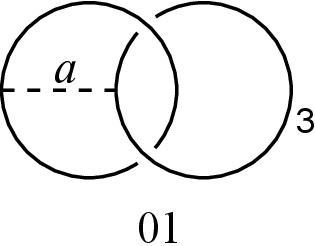}}
\raisebox{45pt} {\parbox[t]{102pt}{$|G|=6$\\$a_{uk}$: II, $g=2$}}
\scalebox{0.4}{\includegraphics*[0pt,0pt][152pt,152pt]{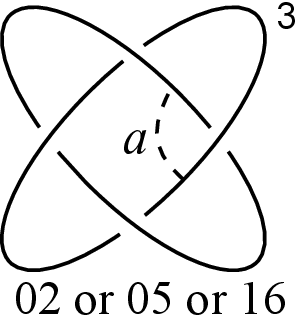}}
\raisebox{45pt} {\parbox[t]{102pt}{$|G|=18$\\$a_{uk}$: II, $g=4$}}
\end{center}

\begin{center}
\scalebox{0.4}{\includegraphics*[0pt,0pt][152pt,152pt]{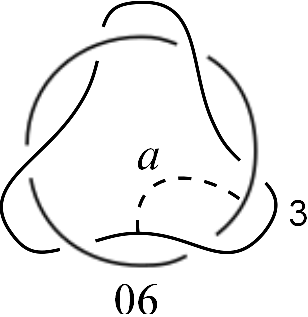}}
\raisebox{45pt} {\parbox[t]{102pt}{$|G|=48$\\$a_{uk}$: II, $g=9$}}
\scalebox{0.4}{\includegraphics*[0pt,0pt][152pt,152pt]{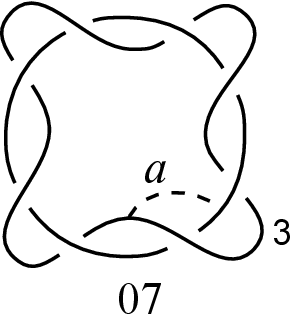}}
\raisebox{45pt} {\parbox[t]{102pt}{$|G|=144$\\$a_{uk}$: II, $g=25$}}
\end{center}

\begin{center}
\scalebox{0.4}{\includegraphics*[0pt,0pt][152pt,152pt]{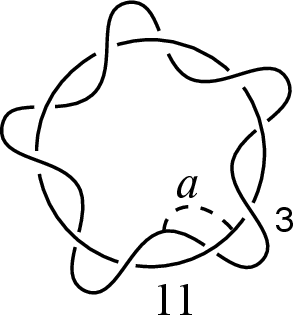}}
\raisebox{45pt} {\parbox[t]{102pt}{$|G|=720$\\$a_{uk}$: II,
$g=121$}}
\scalebox{0.4}{\includegraphics*[0pt,0pt][152pt,152pt]{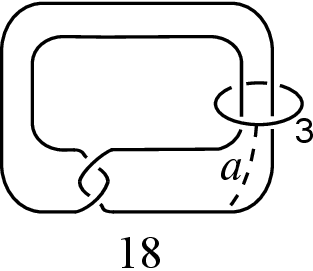}}
\raisebox{45pt} {\parbox[t]{102pt}{$|G|=144$\\$a_{uk}$: II, $g=25$}}
\end{center}

\begin{center}
\scalebox{0.4}{\includegraphics*[0pt,0pt][152pt,152pt]{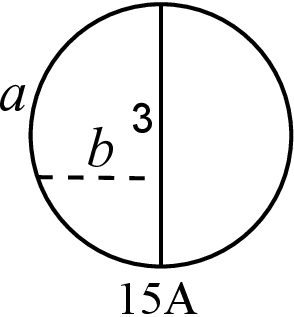}}
\raisebox{45pt} {\parbox[t]{102pt}{$|G|=6$\\$a$: I, $g=2$\\$b_{uk}$:
II, $g=2$}}
\scalebox{0.4}{\includegraphics*[0pt,0pt][152pt,152pt]{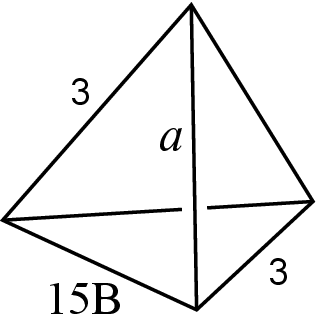}}
\raisebox{45pt} {\parbox[t]{102pt}{$|G|=18$\\$a$: I, $g=4$}}
\end{center}

\begin{center}
\scalebox{0.4}{\includegraphics*[0pt,0pt][152pt,152pt]{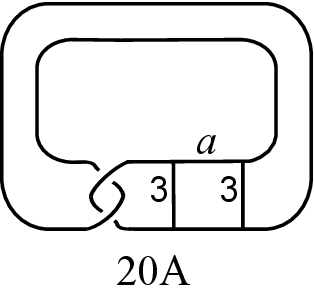}}
\raisebox{45pt} {\parbox[t]{102pt}{$|G|=144$\\$a$: I, $g=25$}}
\scalebox{0.4}{\includegraphics*[0pt,0pt][152pt,152pt]{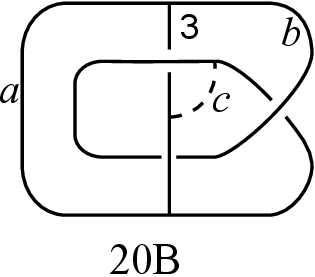}}
\raisebox{45pt} {\parbox[t]{102pt}{$|G|=48$\\$a$: I, $g=9$\\$b_k$:
I, $g=9$\\$c_{uk}$: II, $g=9$}}
\end{center}

\begin{center}
\scalebox{0.4}{\includegraphics*[0pt,0pt][152pt,152pt]{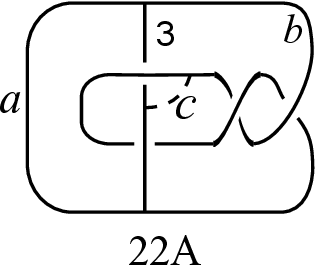}}
\raisebox{45pt} {\parbox[t]{102pt}{$|G|=720$\\$a$: I,
$g=121$\\$b_k$: I, $g=121$\\$c_{uk}$: II, $g=121$}} \hspace*{170pt}
\end{center}

\begin{table}[h]
\caption{Fibred case: type is not $(2, 2, 3,
3)$}\label{tab:fibnot2233}
\end{table}
\begin{center}
\scalebox{0.4}{\includegraphics*[0pt,0pt][152pt,152pt]{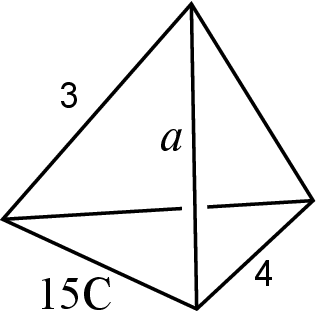}}
\raisebox{45pt} {\parbox[t]{102pt}{$|G|=24$\\$a$:(2,2,3,4), $g=6$}}
\scalebox{0.4}{\includegraphics*[0pt,0pt][152pt,152pt]{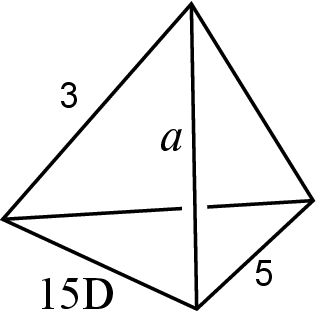}}
\raisebox{45pt} {\parbox[t]{102pt}{$|G|=30$\\$a$:(2,2,3,5), $g=8$}}
\end{center}

\begin{center}
\scalebox{0.4}{\includegraphics*[0pt,0pt][152pt,152pt]{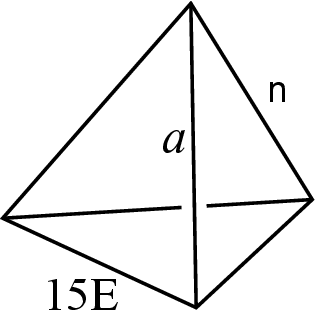}}
\raisebox{45pt}{\parbox[t]{102pt}{$|G|=4n$\\$a$:(2,2,2,n)\\$g=n-1$}}
\scalebox{0.4}{\includegraphics*[0pt,0pt][152pt,152pt]{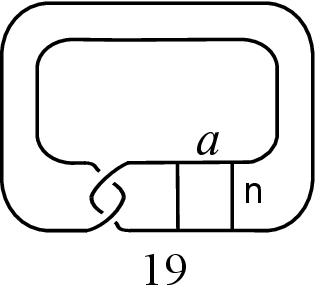}}
\raisebox{45pt}{\parbox[t]{102pt}{$|G|=4n^2$\\$a$:(2,2,2,n)\\$g=(n-1)^2$}}
\end{center}

\begin{center}
\scalebox{0.4}{\includegraphics*[0pt,0pt][152pt,152pt]{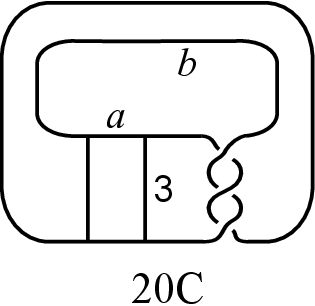}}
\raisebox{45pt} {\parbox[t]{102pt}{$|G|=96$\\$a$:(2,2,2,3),
$g=9$\\$b_{k}$:(2,2,2,3), $g=9$}}
\scalebox{0.4}{\includegraphics*[0pt,0pt][152pt,152pt]{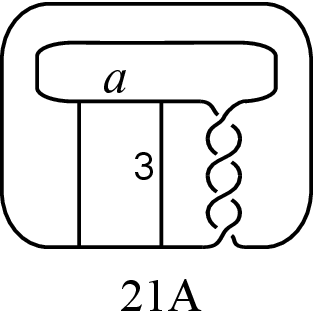}}
\raisebox{45pt} {\parbox[t]{102pt}{$|G|=288$\\$a$:(2,2,2,3),
$g=25$}}
\end{center}

\begin{center}
\scalebox{0.4}{\includegraphics*[0pt,0pt][152pt,152pt]{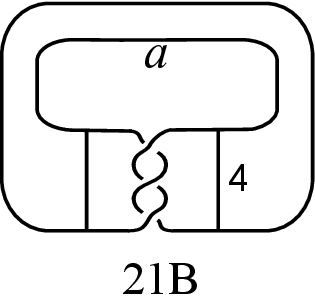}}
\raisebox{45pt} {\parbox[t]{102pt}{$|G|=384$\\$a$:(2,2,2,4),
$g=49$}}
\scalebox{0.4}{\includegraphics*[0pt,0pt][152pt,152pt]{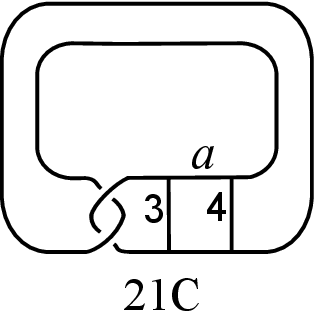}}
\raisebox{45pt} {\parbox[t]{102pt}{$|G|=576$\\$a$:(2,2,3,4),
$g=121$}}
\end{center}

\begin{center}
\scalebox{0.4}{\includegraphics*[0pt,0pt][152pt,152pt]{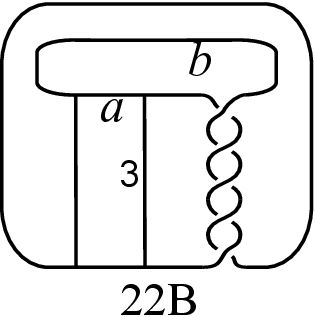}}
\raisebox{45pt} {\parbox[t]{102pt}{$|G|=1440$\\$a$:(2,2,2,3),
$g=121$\\$b_{k}$:(2,2,2,3), $g=121$}}
\scalebox{0.4}{\includegraphics*[0pt,0pt][152pt,152pt]{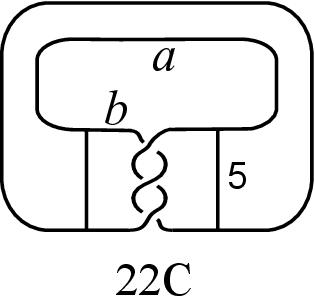}}
\raisebox{45pt} {\parbox[t]{102pt}{$|G|=2400$\\$a$:(2,2,2,5),
$g=361$\\$b_{k}$:(2,2,2,5), $g=361$}}
\end{center}

\begin{center}
\scalebox{0.4}{\includegraphics*[0pt,0pt][152pt,152pt]{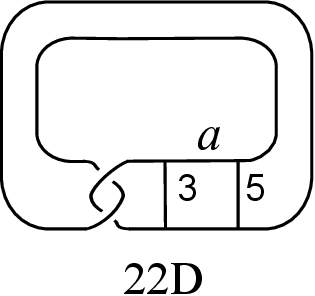}}
\raisebox{45pt} {\parbox[t]{102pt}{$|G|=3600$\\$a$:(2,2,3,5),
$g=841$}}\hspace*{170pt}
\end{center}

\begin{table}[h]
\caption{Non-fibred case}\label{tab:nonfibre}
\end{table}
\begin{center}
\scalebox{0.4}{\includegraphics*[0pt,0pt][152pt,152pt]{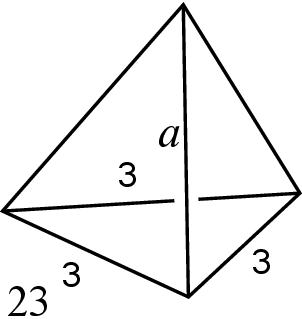}}
\raisebox{45pt} {\parbox[t]{102pt}{$|G|=96$\\$a$: I, $g=17$}}
\scalebox{0.4}{\includegraphics*[0pt,0pt][152pt,152pt]{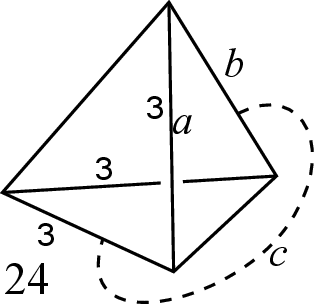}}
\raisebox{45pt} {\parbox[t]{102pt}{$|G|=60$\\$a$:(2,2,2,3),
$g=6$\\$b$: I, $g=11$\\$c_{k}$: II, $g=11$}}
\end{center}

\begin{center}
\scalebox{0.4}{\includegraphics*[0pt,0pt][152pt,152pt]{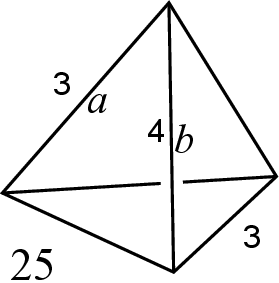}}
\raisebox{45pt} {\parbox[t]{102pt}{$|G|=576$\\$a$:(2,2,2,4),
$g=73$\\$b$: I, $g=97$}}
\scalebox{0.4}{\includegraphics*[0pt,0pt][152pt,152pt]{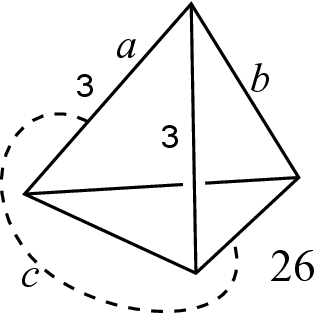}}
\raisebox{45pt} {\parbox[t]{102pt}{$|G|=24$\\$a$:(2,2,2,3),
$g=3$\\$b$: I, $g=5$\\$c_k$: II, $g=5$}}
\end{center}

\begin{center}
\scalebox{0.4}{\includegraphics*[0pt,0pt][152pt,152pt]{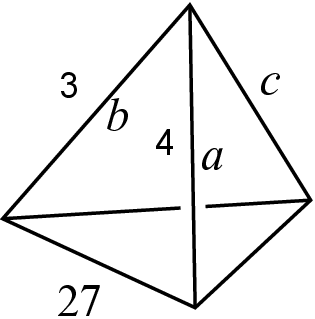}}
\raisebox{45pt} {\parbox[t]{102pt}{$|G|=48$\\$a$:(2,2,2,3),
$g=5$\\$b$:(2,2,2,4), $g=7$\\$c$:(2,2,3,4), $g=11$}}
\scalebox{0.4}{\includegraphics*[0pt,0pt][152pt,152pt]{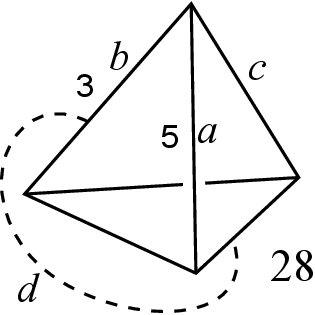}}
\raisebox{45pt} {\parbox[t]{102pt}{$|G|=120$\\$a$:(2,2,2,3),
$g=11$\\$b$:(2,2,2,5), $g=19$\\$c$:(2,2,3,5), $g=29$\\$d_{k}$: II,
$g=21$}}
\end{center}

\begin{center}
\scalebox{0.4}{\includegraphics*[0pt,0pt][152pt,152pt]{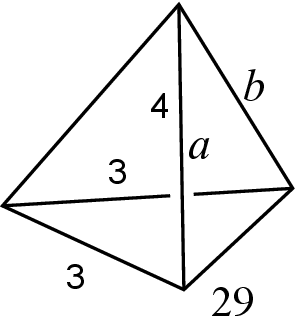}}
\raisebox{45pt} {\parbox[t]{102pt}{$|G|=192$\\$a$:(2,2,2,3),
$g=17$\\$b$:(2,2,3,4), $g=41$}}
\scalebox{0.4}{\includegraphics*[0pt,0pt][152pt,152pt]{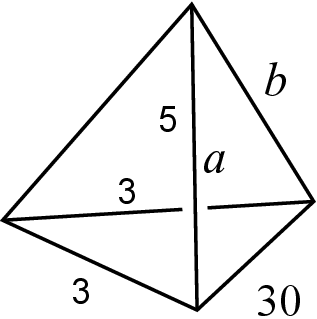}}
\raisebox{45pt} {\parbox[t]{102pt}{$|G|=7200$\\$a$:(2,2,2,3),
$g=601$\\$b$:(2,2,3,5), $g=1681$}}
\end{center}

\begin{center}
\scalebox{0.4}{\includegraphics*[0pt,0pt][152pt,152pt]{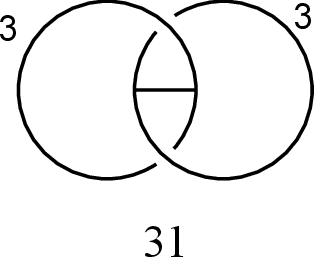}}
\raisebox{45pt} {\parbox[t]{102pt}{$|G|=288$\\No allowable\\
2-suborbifold}}
\scalebox{0.4}{\includegraphics*[0pt,0pt][152pt,152pt]{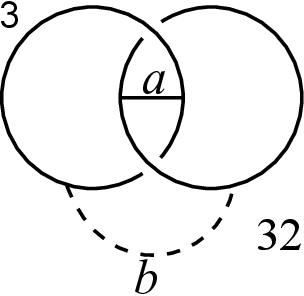}}
\raisebox{45pt} {\parbox[t]{102pt}{$|G|=24$\\$a$: I,
$g=5$\\$b_{uk}$: II, $g=5$}}
\end{center}

\begin{center}
\scalebox{0.4}{\includegraphics*[0pt,0pt][152pt,152pt]{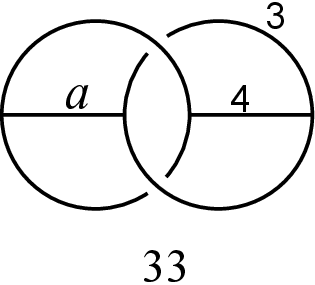}}
\raisebox{45pt} {\parbox[t]{102pt}{$|G|=1152$\\$a$:(2,2,2,3),
$g=97$}}
\scalebox{0.4}{\includegraphics*[0pt,0pt][152pt,152pt]{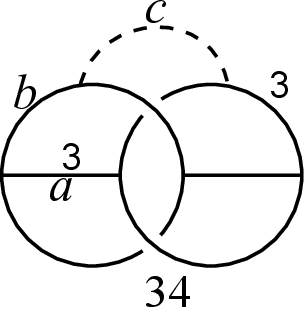}}
\raisebox{45pt} {\parbox[t]{102pt}{$|G|=120$\\$a$:(2,2,2,3),
$g=11$\\$b_{k}$:(2,2,2,3), $g=11$\\$c_{k}$: II, $g=21$}}
\end{center}

\begin{center}
\scalebox{0.4}{\includegraphics*[0pt,0pt][152pt,152pt]{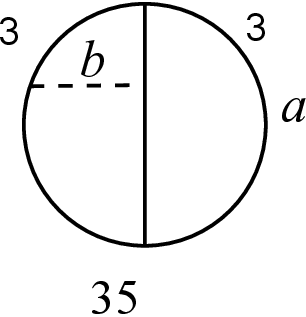}}
\raisebox{45pt} {\parbox[t]{102pt}{$|G|=12$\\$a$: I,
$g=3$\\$b_{uk}$: II, $g=3$}}
\scalebox{0.4}{\includegraphics*[0pt,0pt][152pt,152pt]{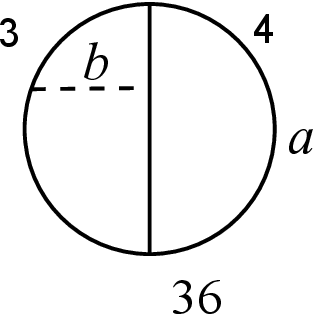}}
\raisebox{45pt} {\parbox[t]{102pt}{$|G|=24$\\$a$: I,
$g=5$\\$b_{uk}$: II, $g=5$}}
\end{center}

\begin{center}
\scalebox{0.4}{\includegraphics*[0pt,0pt][152pt,152pt]{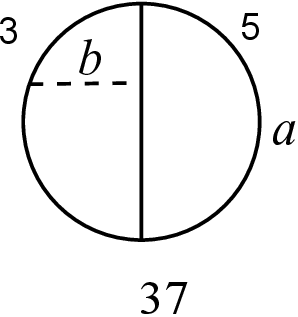}}
\raisebox{45pt} {\parbox[t]{102pt}{$|G|=60$\\$a$: I,
$g=11$\\$b_{uk}$: II, $g=11$}}
\scalebox{0.4}{\includegraphics*[0pt,0pt][152pt,152pt]{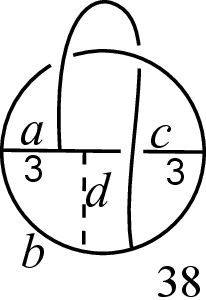}}
\raisebox{45pt} {\parbox[t]{102pt}{$|G|=2880$\\$a, b_k,
c_k$:(2,2,2,3)\\$g=241$\\$d_k$: II, $g=481$}}
\end{center}

\begin{center}
\scalebox{0.4}{\includegraphics*[0pt,0pt][152pt,152pt]{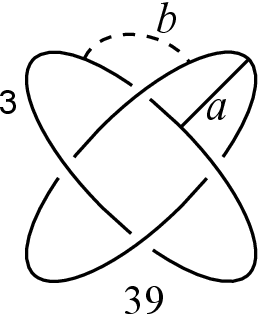}}
\raisebox{45pt} {\parbox[t]{102pt}{$|G|=576$\\$a$: I,
$g=97$\\$b_{uk}$: II, $g=97$}}
\scalebox{0.4}{\includegraphics*[0pt,0pt][152pt,152pt]{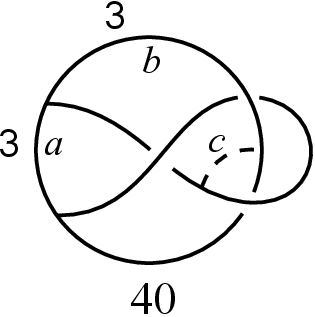}}
\raisebox{45pt} {\parbox[t]{102pt}{$|G|=1440$\\$a$: I,
$g=241$\\$b_k$: I, $g=241$\\$c_{uk}$: II, $g=241$}}
\end{center}


\section{Edges in orbifolds provide  $(S^3, \Gamma)$ with maximum symmetry}
\subsection{Maximum orders of symmetries on $(S^3, \Gamma)$}

We first prove Theorem \ref{mg}.
The following primary fact is used repeatedly and implicitly in the proof.

\begin{proposition} \label{SO(4)}
Suppose $G$ is a finite group of $SO(4)$ acting on $(S^3, T)$. Here
$T$ is a polyhedron which can not be embedded into a circle (in particular, $T$ may be either a surface, or a handlebody, or
a graph with $g>1$). Then the restriction of $G$ on $T$ is faithful.
\end{proposition}

\begin{proof} Suppose $g\in G$ and the restriction of its action on $T$ is the identity.
As an orientation preserving isometry, its
fixed point set $\text{Fix}(g)$ is either the empty set, or a
circle, or the whole $S^3$. Since $T\subset  \text{Fix}(g)$ and $T$
is not a subset of a circle, we have $\text{Fix}(g)=S^3$, and hence
$g$ is the identity of $G\subset SO(4)$.
\end{proof}

\begin{proof}[Proof of Theorem \ref{mg}]
We can assume that $S^3$ has the standard spherical geometry and
$G\subset SO(4)$. Suppose $G$ acts on $(S^3, \Gamma)$ for some graph
$\Gamma\subset S^3$ of genus $g$. Let $N_\epsilon(\Gamma)=\{x \in
S^3\mid dist(x, \Gamma)\leq \epsilon\}$ be the
$\epsilon$-neighborhood of $\Gamma$. When $\epsilon>0$ is
sufficiently small, $N_\epsilon(\Gamma)$ is a handlebody of genus
$g$. Since $G$ acts isometrically, $N_\epsilon(\Gamma)$ is invariant
under the group action. Notice that generally $\partial
N_\epsilon(\Gamma)$ is not smooth. But we can choose a smaller
equivariant neighborhood $U_\epsilon$ of $\Gamma$ such that
$\partial U_\epsilon$ is a smooth submanifold in $S^3$ and $\partial
U_\epsilon\simeq\Sigma_g$ (see \cite[Remark 3.3]{WWZ}). Hence we
have a $G$-action on $(S^3, \Sigma_g)$, and it follows that $m_g\le
OE_g$.

Suppose $G$ acts on $(S^3, \Sigma_g)$ for some $\Sigma_g\subset S^3$
and $|G|>4(g-1)$. Then by \cite[Proposition 2.5]{WWZZ}, $\Sigma_g$ bounds an
handlebody $V_g$ in $S^3$. Since $G$ acts orientation preservingly
on both $S^3$ and $\Sigma_g$, $V_g$ is invariant under the group
action, and moreover by \cite[Lemma 2.9]{WWZZ} $V_g/G$ is a $N(a)$, the regular neighborhood
of an allowable singular edge/dashed arc in an orbifold
$\mathcal{O}=S^3/G$ listed in Theorem \ref{classify}. Let $p: S^3
\to \mathcal{O}$ be the branched covering. Then $\Gamma=p^{-1}(a)$
is a connected graph which is invariant under the $G$-action. $V_g$
is a regular neighborhood of $\Gamma$, therefore the genus of
$\Gamma$ is $g$. Since $OE_g>4(g-1)$, we have $OE_g\le m_g$.

Hence $m_g=OE_g>4(g-1)$. We have proved above that if $|G|>4(g-1)$,
then $m_g$ is realized by an action on an unknotted (resp. knotted) graph if and only if
$OE_g$ is realized by an action on an unknotted (resp. knotted) surface. Therefore we can
copy [WWZZ, Theorem 1.1] as the remaining part of Theorem \ref{mg}.
\end{proof}



\subsection{Edges in orbifolds provide  $(S^3, \Gamma)$  with maximum symmetry}

Let us recall some facts about graph of groups
before further discussion. For more details, see \cite{Se}.
Suppose we have an extendable $G$-action on $\Gamma$ with respect to
an embedding $e: \Gamma\hookrightarrow S^3$. Then we have an
embedding $e/G: \Gamma/G\hookrightarrow S^3/G$. Here $\Gamma/G$ can
be thought as a graph of groups or a `graph orbifold'.

For a vertex $v$ of $\Gamma/G$, let $v'$ be one of its pre-image in
$\Gamma$. Then the \textit{vertex group} $G_v$ can be identified to the
stable subgroup of $v'$. Similarly for an edge $e$ of $\Gamma/G$ we
have the \textit{edge group} $G_e$. The \textit{index} of $v$ (resp. $e$) can be defined
to be $|G_v|$ (resp. $|G_e|$). If $v$ is a vertex of $e$, then there is
a natural injection from $G_e$ to $G_v$. If we forget the groups and
injections, we get the \textit{underlying graph} of $\Gamma/G$, denoted by
$|\Gamma/G|$.

We can define the \textit{Euler characteristic} of $\Gamma/G$ by
$$\chi(\Gamma/G)=\sum 1/|G_v| - \sum 1/|G_e|.  \qquad (3.1)$$
Here the first sum consists of all vertices of $\Gamma/G$, and the
second sum consists of all edges of $\Gamma/G$. If $G$ is trivial, then we get the Euler characteristic $\chi(\Gamma)$ of $\Gamma$.
Generally, by multiplying $|G|$ to the two sides of $(3.1)$ (or directly by just counting the number $a$ of vertices and $b$ of edges of the graph $\Gamma$ whose Euler characteristic is $a-b$,  considering the number of elements in each orbit and the order of a stabilizer of an element in the orbit), we have
$$\chi(\Gamma/G)=\chi(\Gamma)/|G|.\qquad  (3.2)$$


\begin{lemma}\label{allowable}
Suppose there is an extendable $G$-action on $\Gamma$, $\Gamma$ is
minimal and $|G|>4(g-1)$, then $\Gamma/G$ is an allowable singular
edge/dashed arc.
\end{lemma}

\begin{proof}
By (3.1), it is easy to derive that
$$\chi(\Gamma/G)=\chi(|\Gamma/G|)-\sum\nolimits (1-1/|G_v|) +
\sum\nolimits (1-1/|G_e|).\qquad (3.3)$$
Here the first sum consists of all
vertices of $\Gamma/G$ with indices bigger than 1; The second sum
consists of all edges of $\Gamma/G$ with indices bigger than 1.

Because in our case each $G_v$ is a finite subgroup of $SO(3)$, we
have
$$(1-1/|G_v|) = \sum\nolimits_v(1-1/|G_e|)/2.\qquad (3.4)$$
Here the sum consists of all edges of $\Theta$ containing $v$ as a
vertex. Recall that $\Theta$ is the singular trivalent graph of
$S^3/G$. Then combining (3.4)  with (3.3) we have
$$\chi(\Gamma/G)=\chi(|\Gamma/G|)-\sum\nolimits'(1-1/|G_e|)/2.\qquad (3.5)$$
Here the sum consists of all edges of $\Theta$ which are not edges
of $\Gamma/G$ but contain vertices of $\Gamma/G$.

Since $g\geq 2$, $\chi(\Gamma)=1-g \leq -1$. Then since $|G|>4(g-1)$, applying (3.2) we have $-1/4<\chi(\Gamma/G)<0$. Clearly every
item in $\sum'$ is not smaller than $1/4$. Hence we must have
$\chi(|\Gamma/G|)=1$ and the sum $\sum'$ contains at most $4$
items. Then we know that $|\Gamma/G|$ is a tree. Since $\Gamma$
contains no free edge, every leaf of $\Gamma/G$ gives at least two
items in $\sum'$. Then $\Gamma/G$ has exactly two leaves and all
the other vertices have degree 2. Since $\Gamma$ contains no extra
vertices, every vertex of $\Gamma/G$ other than leaves gives at
least one item in $\sum'$. Hence there is no such vertex in
$\Gamma/G$ and $\Gamma/G$ contains only one edge. Clearly the boundary of its regular neighborhood is allowable, so $\Gamma/G$ is an allowable singular edge
or dashed arc.
\end{proof}

Hence by Theorem \ref{mg} and Lemma \ref{allowable}, to find all
maximum symmetry actions on $S^3$ (leaving some graph invariant), we have only to find all (allowable)
singular edges/dashed arcs in spherical $3$-orbifolds. This is very close to the information given by Theorem \ref{classify} in [WWZZ, Theorem 6.1], but we need to modify the definition of equivalence previously given.

In the graph case,
call two singular edges/dashed arcs \textit{equivalent} if there is an
orbifold automorphism sending one to the other. However, if two (nonequivalent) singular edges/dashed arcs have regular neighborhoods with isotopic boundaries (so that an unknotted $2$-orbifold splits the spherical $3$-orbifold into two handlebody orbifolds), then we say that the edges/arcs are \textit{dual} to each other.


In the study of maximum symmetry of $(S^3, \Sigma_g)$, it is
reasonable to call two dual edges equivalent, since they produce
the same allowable 2-orbifold. But it is not good to call them
equivalent in the study of maximum symmetry of $(S^3, \Gamma)$,
since they may be different graph orbifolds or correspond to
different graphs.

Then by a routine checking of Theorem \ref{classify}, as well as
Lemma \ref{allowable}, we have the following Theorem \ref{position}
for our further study of the realizations of the maximum symmetry of
$(S^3, \Gamma)$ in the following sense: (1) we only pick information
from Theorem \ref{classify} related to $m_g$ (so among 40 orbifolds
listed in Theorem \ref{classify}, only 17 orbifolds appear in
Theorem \ref{position});  and list the orbifolds according to the
sizes of $m_g$ (so one figure in Theorem \ref{classify} can become
several figures in Theorem \ref{position}, for example, 15E or 27). (2)
We mark a pair of edges by $a$ and $a'$ and so on if they are dual.


\begin{theorem}\label{position}
1. If $g \neq 21, 481$, then $\Gamma$ is a MS graph if and only if
$\Gamma/G\hookrightarrow S^3/G$ belongs to the following table,
labeled by $a, a', b, c, \cdots$  (up to automorphisms of $S^3/G$).

2. When $g=21, 481$, $\Gamma$ is a MS graph if and only if
$\Gamma/G$ is an allowable dashed arc in orbifolds listed in
$|G|=6(g-1)$, and we just give one possible $\Gamma/G$.

\end{theorem}
\begin{table}[h]
\caption{Allowable edges/arcs corresponding to
MS graphs}\label{tab:allow}
\end{table}
\centerline{$|G|=12(g-1)$}
\begin{center}
\scalebox{0.4}{\includegraphics*[0pt,0pt][152pt,152pt]{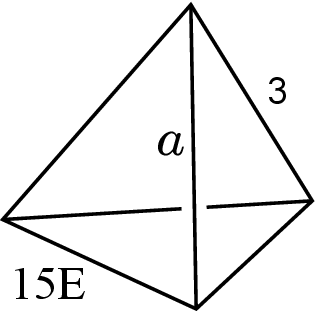}}
\hspace{10pt}\raisebox{35pt} {\parbox[t]{102pt}{$|G|=12$\\$g=2$}}
\scalebox{0.4}{\includegraphics*[0pt,0pt][152pt,152pt]{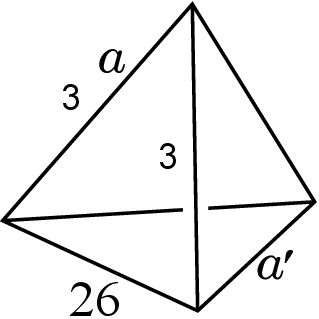}}
\hspace{10pt}\raisebox{35pt} {\parbox[t]{102pt}{$|G|=24$\\$g=3$}}
\end{center}

\begin{center}
\scalebox{0.4}{\includegraphics*[0pt,0pt][152pt,152pt]{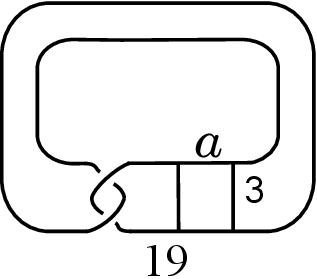}}
\hspace{10pt}\raisebox{35pt} {\parbox[t]{102pt}{$|G|=36$\\$g=4$}}
\scalebox{0.4}{\includegraphics*[0pt,0pt][152pt,152pt]{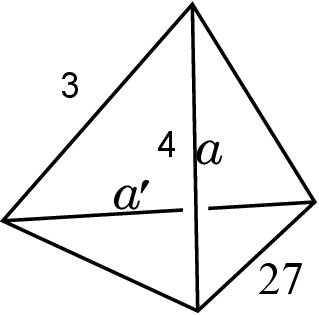}}
\hspace{10pt}\raisebox{35pt} {\parbox[t]{102pt}{$|G|=48$\\$g=5$}}
\end{center}

\begin{center}
\scalebox{0.4}{\includegraphics*[0pt,0pt][152pt,152pt]{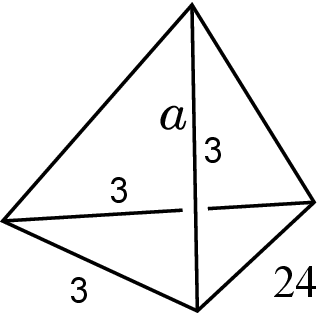}}
\hspace{10pt}\raisebox{35pt} {\parbox[t]{102pt}{$|G|=60$\\$g=6$}}
\scalebox{0.4}{\includegraphics*[0pt,0pt][152pt,152pt]{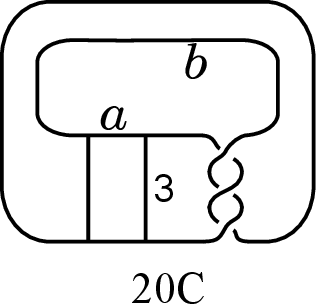}}
\hspace{10pt}\raisebox{35pt} {\parbox[t]{102pt}{$|G|=96$\\$g=9$}}
\end{center}

\begin{center}
\scalebox{0.4}{\includegraphics*[0pt,0pt][152pt,152pt]{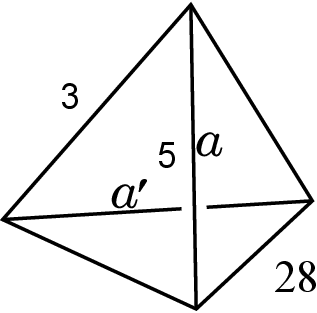}}
\hspace{10pt}\raisebox{35pt} {\parbox[t]{102pt}{$|G|=120$\\$g=11$}}
\scalebox{0.4}{\includegraphics*[0pt,0pt][152pt,152pt]{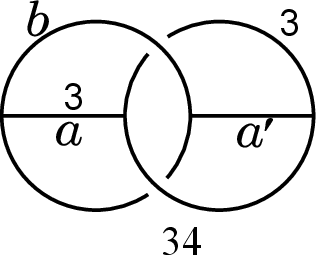}}
\hspace{10pt}\raisebox{35pt} {\parbox[t]{102pt}{$|G|=120$\\$g=11$}}
\end{center}

\begin{center}
\scalebox{0.4}{\includegraphics*[0pt,0pt][152pt,152pt]{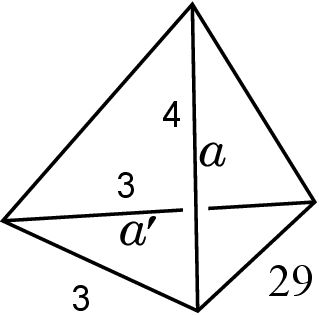}}
\hspace{10pt}\raisebox{35pt} {\parbox[t]{102pt}{$|G|=192$\\$g=17$}}
\scalebox{0.4}{\includegraphics*[0pt,0pt][152pt,152pt]{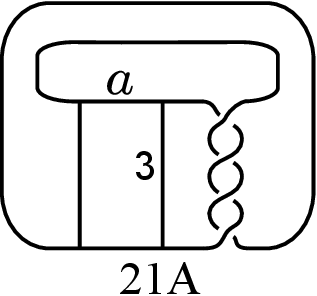}}
\hspace{10pt}\raisebox{35pt} {\parbox[t]{102pt}{$|G|=288$\\$g=25$}}
\end{center}

\begin{center}
\scalebox{0.4}{\includegraphics*[0pt,0pt][152pt,152pt]{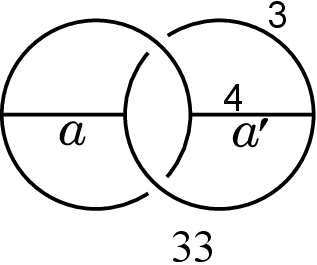}}
\hspace{10pt}\raisebox{35pt} {\parbox[t]{102pt}{$|G|=1152$\\$g=97$}}
\scalebox{0.4}{\includegraphics*[0pt,0pt][152pt,152pt]{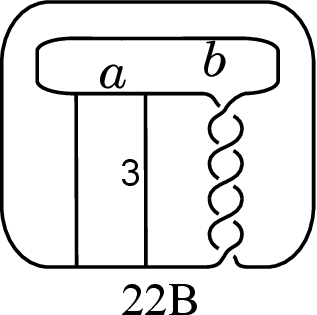}}
\hspace{10pt}\raisebox{35pt}
{\parbox[t]{102pt}{$|G|=1440$\\$g=121$}}
\end{center}

\begin{center}
\scalebox{0.4}{\includegraphics*[0pt,0pt][152pt,152pt]{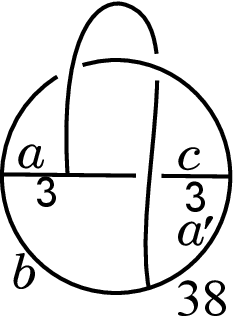}}
\hspace{10pt}\raisebox{35pt}
{\parbox[t]{102pt}{$|G|=2880$\\$g=241$}}
\scalebox{0.4}{\includegraphics*[0pt,0pt][152pt,152pt]{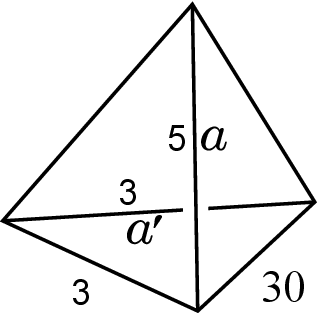}}
\hspace{10pt}\raisebox{35pt}
{\parbox[t]{102pt}{$|G|=7200$\\$g=601$}}
\end{center}

\centerline{$|G|=8(g-1)$}
\begin{center}
\scalebox{0.4}{\includegraphics*[0pt,0pt][152pt,152pt]{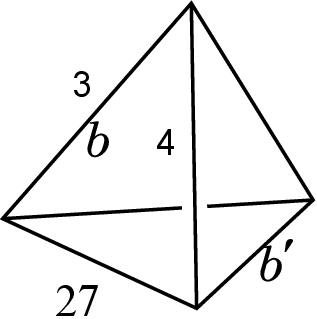}}
\hspace{10pt}\raisebox{35pt} {\parbox[t]{102pt}{$|G|=48$\\$g=7$}}
\scalebox{0.4}{\includegraphics*[0pt,0pt][152pt,152pt]{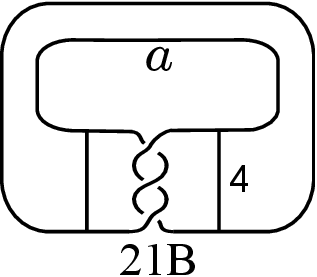}}
\hspace{10pt}\raisebox{35pt} {\parbox[t]{102pt}{$|G|=384$\\$g=49$}}
\end{center}

\begin{center}
\scalebox{0.4}{\includegraphics*[0pt,0pt][152pt,152pt]{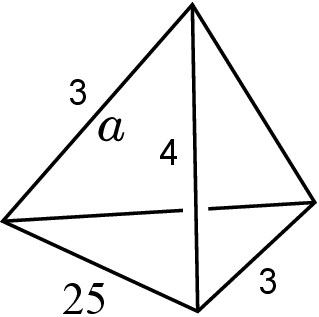}}
\hspace{10pt}\raisebox{35pt}
{\parbox[t]{102pt}{$|G|=576$\\$g=73$}}\hspace*{180pt}
\end{center}

\centerline{$|G|=20(g-1)/3$}
\begin{center}
\scalebox{0.4}{\includegraphics*[0pt,0pt][152pt,152pt]{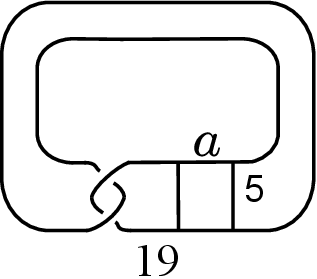}}
\hspace{10pt}\raisebox{35pt} {\parbox[t]{102pt}{$|G|=100$\\$g=16$}}
\scalebox{0.4}{\includegraphics*[0pt,0pt][152pt,152pt]{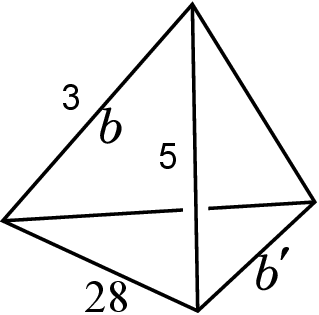}}
\hspace{10pt}\raisebox{35pt} {\parbox[t]{102pt}{$|G|=120$\\$g=19$}}
\end{center}

\begin{center}
\scalebox{0.4}{\includegraphics*[0pt,0pt][152pt,152pt]{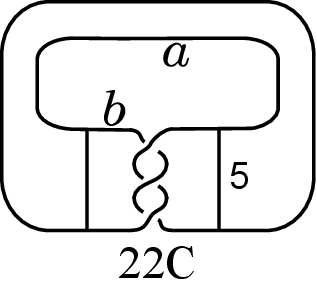}}
\hspace{10pt}\raisebox{35pt}
{\parbox[t]{102pt}{$|G|=2400$\\$g=361$}}\hspace*{180pt}
\end{center}

\centerline{$|G|=6(g-1)$}
\begin{center}
\scalebox{0.4}{\includegraphics*[0pt,0pt][152pt,152pt]{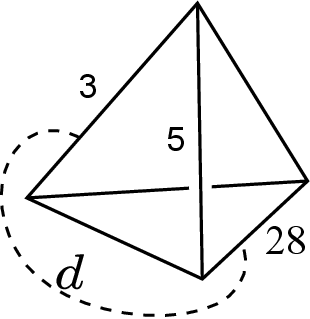}}
\hspace{10pt}\raisebox{35pt} {\parbox[t]{102pt}{$|G|=120$\\$g=21$}}
\scalebox{0.4}{\includegraphics*[0pt,0pt][152pt,152pt]{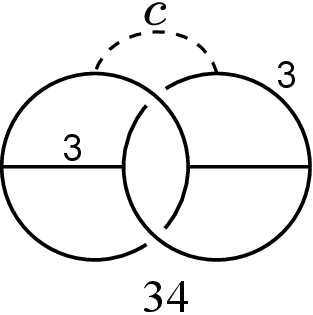}}
\hspace{10pt}\raisebox{35pt} {\parbox[t]{102pt}{$|G|=120$\\$g=21$}}
\end{center}

\begin{center}
\scalebox{0.4}{\includegraphics*[0pt,0pt][152pt,152pt]{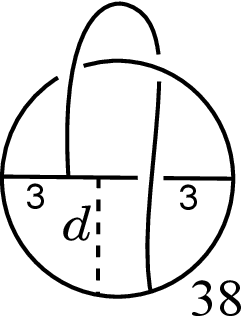}}
\hspace{10pt}\raisebox{35pt}
{\parbox[t]{102pt}{$|G|=2880$\\$g=481$}}\hspace*{180pt}
\end{center}

\centerline{$|G|=24(g-1)/5$ and $30(g-1)/7$}
\begin{center}
\scalebox{0.4}{\includegraphics*[0pt,0pt][152pt,152pt]{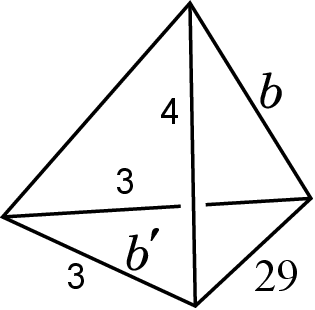}}
\hspace{10pt}\raisebox{35pt} {\parbox[t]{102pt}{$|G|=192$\\$g=41$}}
\scalebox{0.4}{\includegraphics*[0pt,0pt][152pt,152pt]{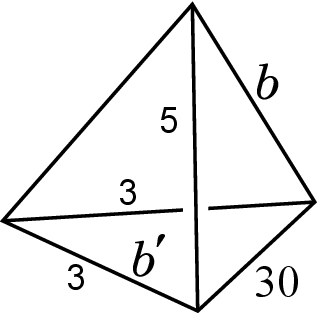}}
\hspace{10pt}\raisebox{35pt}
{\parbox[t]{102pt}{$|G|=7200$\\$g=1681$}}
\end{center}

\centerline{$|G|=4(\sqrt{g}+1)^2$, $g=k^2, k\neq3, 5, 7, 11, 19,
41$}
\begin{center}
\scalebox{0.4}{\includegraphics*[0pt,0pt][152pt,152pt]{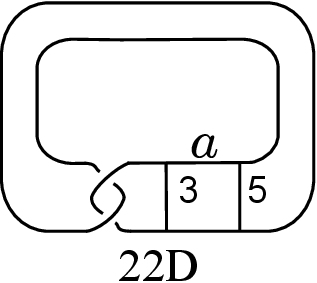}}
\hspace{10pt}\raisebox{35pt}
{\parbox[t]{102pt}{$|G|=3600$\\$g=841$}}
\scalebox{0.4}{\includegraphics*[0pt,0pt][152pt,152pt]{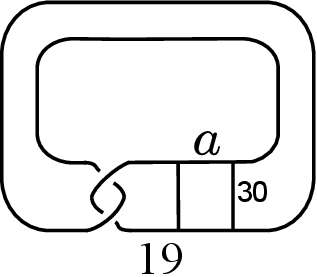}}
\hspace{10pt}\raisebox{35pt}
{\parbox[t]{102pt}{$|G|=3600$\\$g=841$}}
\end{center}

\begin{center}
\scalebox{0.4}{\includegraphics*[0pt,0pt][152pt,152pt]{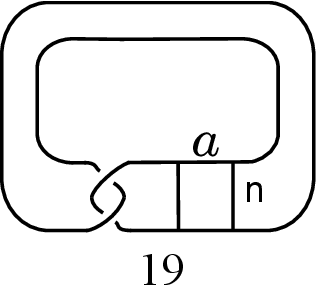}}
\hspace{10pt}\raisebox{35pt}
{\parbox[t]{102pt}{Others\\$g=(n-1)^2$}}\hspace*{180pt}
\end{center}

\centerline{$|G|=4(g+1)$, $g$ belongs to remaining numbers}
\begin{center}
\scalebox{0.4}{\includegraphics*[0pt,0pt][152pt,152pt]{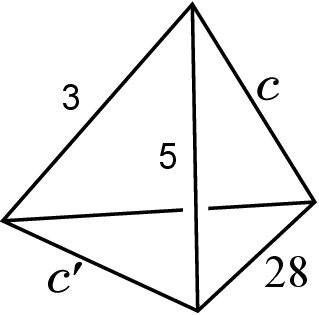}}
\hspace{10pt}\raisebox{35pt} {\parbox[t]{102pt}{$|G|=120$\\$g=29$}}
\scalebox{0.4}{\includegraphics*[0pt,0pt][152pt,152pt]{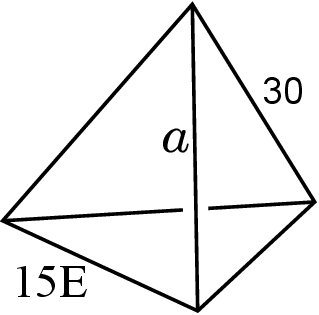}}
\hspace{10pt}\raisebox{35pt} {\parbox[t]{102pt}{$|G|=120$\\$g=29$}}
\end{center}

\begin{center}
\scalebox{0.4}{\includegraphics*[0pt,0pt][152pt,152pt]{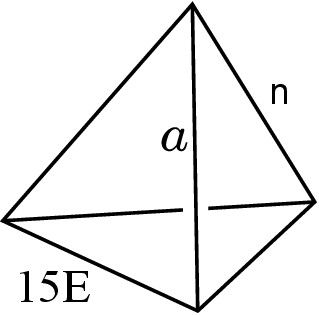}}
\hspace{10pt}\raisebox{35pt}
{\parbox[t]{102pt}{Others\\$g=n-1$}}\hspace*{180pt}
\end{center}

\begin{remark}
1. Theorem \ref{mg} is also valid for the handlebody case, just
considering the handlebody $U_\epsilon(\Gamma)$ and handlebody
orbifold $U_\epsilon(\Gamma)/G$ defined in the proof of Theorem
\ref{mg}.

2. In either the graph case (for actions on $S^3$ leaving an embedded graph invariant) or the handlebody case (for actions on $S^3$ leaving an embedded handlebody invariant), we can also consider the maximum order problem for all the unknotted (resp. all the knotted) embeddings. The results will be completely the same as in the
surface case; see Theorem 1.2 (resp. Theorem 1.3) of \cite{WWZZ}.
\end{remark}

\section {Abstract and spatial MS graphs}
To give the detailed classification of MS graphs, MS spatial graphs
and the group actions, we will use some notations: let
$\mathcal{O}_{N}$ denote the number ``$N$'' orbifold in Theorem
\ref{position}. If $\gamma$ denotes an allowable singular
edge/dashed arc in $\mathcal{O}_N$, let $\Gamma_N^\gamma$ denote the
($G$-invariant MS) graph in $S^3$ which is the pre-image of $\gamma$ in $S^3$. We often put more information on $\Gamma_N^\gamma$,
writing $\Gamma_N^\gamma(g)$ if $\Gamma_N^\gamma$ has genus $g$,
or $\Gamma_N^\gamma(g, k)$ if $\Gamma_N^\gamma$ is also knotted.

\subsection{To picture  and to identify/distinguish  MS graphs}\label{pictureG}

This subsection serves as preparation for the proof of our detailed classification in the next subsection.

\noindent {\bf To picture  MS graphs:} Suppose there is an
extendable $G$-action on $\Gamma$, such that $\Gamma/G$ is a
singular edge/dashed arc in $S^3/G$. Then $\pi_1(\Gamma/G)\cong
\pi(U_\epsilon/G)$ and the pre-image of $U_\epsilon/G$ in $S^3$ is
connected. Hence the embedding $U_\epsilon/G\hookrightarrow S^3/G$
induces a surjection on fundamental groups, see \cite[Corollary 2.12]{WWZZ}. And we
have a representation $\phi: \pi_1(\Gamma/G)\twoheadrightarrow G$.
Notice that $\phi$ can be given by Wirtinger presentations, see
\cite[Proposition 2.8]{BMP} or \cite[Example 6.3]{WWZZ}.

\noindent{\bf Claim.} {\it $\Gamma$ can be reconstructed from $\phi$
as following.}

Give $\Gamma/G$ an orientation; denote the start point by $A$, the
terminal point by $B$ and the edge by $e$. Denote the corresponding
vertex groups and edge group by $G_A$, $G_B$ and $G_e$. Then
$\pi_1(\Gamma/G)\cong G_A\ast_{G_e} G_B$, and $\phi$ is injective on
$G_A$, $G_B$ and $G_e$. Hence $\phi(G_A)\cong G_A$, $\phi(G_B)\cong
G_B$, $\phi(G_e)\cong G_e$, and $\phi(G_e)$ lies in the intersection
of $\phi(G_A)$ and $\phi(G_B)$.

Let $\mathcal{S}$ denote $\{\epsilon_g := [a_g, b_g] \mid g \in G\}$, a set of oriented edges indexed by elements of $G$. Then $\phi(G_A)$
(resp. $\phi(G_B)$, $\phi(G_e)$) acts on $\mathcal{S}$ by left
multiplication of subscripts. Start points are identified if they are in the same $G_A$-orbit (resp. terminal points in the same $G_B$-orbit, edges in the same $G_e$-orbit). Since $\phi$ is surjective, we get a connected graph which is the
`covering' of $\Gamma/G$. On it there is a $G$-action via the right
multiplication.

We can use the above construction to write a computer program for
\cite{GAP} and \cite{Mathematica}. Then the graph can be pictured.

\begin{example}\label{graphdraw}
Using computer to picture $\Gamma^{a'}_{34}(11)$.

A presentation of $\mathcal{O}_{34}$ is given in the proof of Theorem  \ref{detailed-classification}.
Firstly run the following codes in \cite{GAP},  we can get
a list of arrows. Odd numbers correspond to the start points and
even numbers correspond to the terminal points.

\begin{lstlisting}
 f:=FreeGroup("x","y","z");
 x:=f.1;
 y:=f.2;
 z:=f.3;
 G:=f/[x^2, y^3, z^2, (z*y)^2, (y*x*z)^2, (y*x*z*x)^3];
 x:=G.1; # now x is an element in G
 y:=G.2;
 z:=G.3;
 GA:=GroupWithGenerators([x,z*y]); # vertex group
 GB:=GroupWithGenerators([z*y,y]);
 Ge:=GroupWithGenerators([z*y]);   # edge group
 l:=RightCosets(G,GA); # right cosets of vertex groups
 r:=RightCosets(G,GB);
 for i in [1..Size(l)]
 do for j in [1..Size(r)]
   do if (Size(Intersection2(l[i],r[j]))<>0) then
     for k in [1..Size(Intersection2(l[i],r[j]))/Size(Ge)]
     do Print("(",2*i-1,"->",2*j,"),");
     od;       # the intersection of a coset pair is
     fi;       # the union of cosets of the edge group
   od;
 od;
 Print("\n");
\end{lstlisting}

Then copy the list to \cite{Mathematica}, and run the following
codes.

\begin{lstlisting}
 <<Combinatorica`
 Needs["GraphUtilities`"]
 O34a'={(1->2),(1->4),(3->2),(3->6),(5->2),(5->8),(7->4),
  (7->10),(9->4),(9->12),(11->6),(11->14),(13->8),(13->16),
  (15->6),(15->18),(17->8),(17->20),(19->10),(19->22),
  (21->12),(21->24),(23->10),(23->26),(25->12),(25->28),
  (27->14),(27->30),(29->16),(29->32),(31->18),(31->28),
  (33->20),(33->26),(35->14),(35->22),(37->16),(37->24),
  (39->18),(39->32),(41->20),(41->30),(43->22),(43->34),
  (45->24),(45->36),(47->26),(47->36),(49->28),(49->34),
  (51->30),(51->38),(53->32),(53->38),(55->34),(55->40),
  (57->36),(57->40),(59->38),(59->40)}
 G34a'=SetGraphOptions[ToCombinatoricaGraph[O34a'],
   EdgeDirection->False]
 GraphPlot[G34a']
\end{lstlisting}

Finally we will get the picture as in Figure \ref{fig:Gr11d}.

\begin{figure}[h]
\centerline{\scalebox{0.6}{\includegraphics{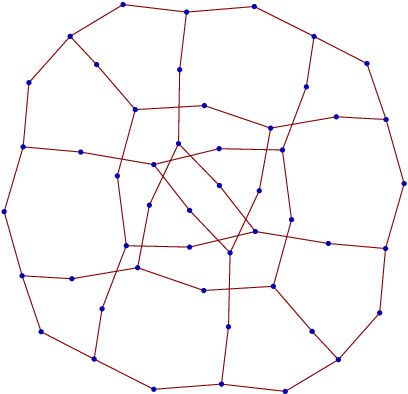}}}
\caption{The MS graph $\Gamma^{a'}_{34}(11)$}\label{fig:Gr11d}
\end{figure}
\end{example}

\noindent {\bf To identify/distinguish  MS graphs:}

A main  method to identify graphs is the following.

Suppose that $\Gamma_i$ are the
($G_i$-invariant MS) graphs in $S^3$, and representations $\phi_i: \pi_1(\Gamma_i/G_i)\cong G_{iA_i}\ast_{G_{ie_i}} G_{iB_i}\to G_i$ are induced by the orbifold embeddings, $i=1,2$.

$$\xymatrix{
  G_{1A_1}\ast_{G_{1e_1} }G_{1B_1} \ar[d]_{\eta} \ar[r]^{\quad\quad\phi_1} & G_1 \ar[d]^{\psi} \ar[r]
  & 1 \\
  G_{2A_2}\ast_{G_{2e_2}} G_{2B_2}  \ar[r]^{\quad\quad\phi_2} & G_2 \ar[r] & 1}
$$

If we have two isomorphisms $\eta$ and $\psi$ in the diagram above to make it
commutative, where $\eta$ maps $G_{1A_1}$ to $G_{2A_2}$,
$G_{1B_1}$ to $G_{2B_2}$, and $G_{1e_1}$ to $G_{2e_2}$,
then clearly $\Gamma_1$ and  $\Gamma_2$ are $G$-equivariant (abstract) graphs.

In practice, we will first give the presentations of the groups and then give the  map between group generators and check if the map gives us a required group isomorphism.



We distinguish graphs by computing the graph invariants such as the
number of vertices or edges, the degree of a vertex, the diameter of
a graph and the girth (the length of a minimal loop) of a graph.

For complicated  cases, we  need computer programs to determine if
maps are isomorphisms and to compute  graph  invariants.

\begin{example}
The map $(\overline{u},\overline{u}_l,\overline{u}_r)\mapsto
(\overline{v},\overline{v}_l,\overline{v}_r)$ in the proof of
Theorem \ref{detailed-classification} is an isomorphism from
$\pi_1(\mathcal{O}_{28})$ to $\pi_1(\mathcal{O}_{34})$.

\begin{lstlisting}
 f:=FreeGroup("x","y","z");
 x:=f.1;
 y:=f.2;
 z:=f.3;
 O28:=f/[x^5, y^2, z^2, (x*z)^3, (x*y)^2, (y*z^(-1))^2];
 O34:=f/[x^2, y^3, z^2, (z*y)^2, (y*x*z)^2, (y*x*z*x)^3];
 x:=O28.1; # x is an element in O28
 y:=O28.2;
 z:=O28.3;
 r:=O34.1; # r is an element in O34
 s:=O34.2;
 t:=O34.3;
 iso28:=IsomorphismPermGroup(O28); # pass to permutation group
 iso34:=IsomorphismPermGroup(O34);
 G28:=Image(iso28);
 G34:=Image(iso34);
 u:=Image(iso28,x*y*z^(-1)*x^(-1));
 ul:=Image(iso28,x*z*x^(-1));
 ur:=Image(iso28,x*z);
 v:=Image(iso34,r^(-1));
 vl:=Image(iso34,s^(-1)*t^(-1));
 vr:=Image(iso34,t^(-1)r^(-1)s^(-1)r);
 GroupHomomorphismByImages(G28,G34,[u,ul,ur],[v,vl,vr]);
\end{lstlisting}

Run the above codes in \cite{GAP}. If it is an isomorphism, then
\cite{GAP} will give the correspondence between elements of the two
groups. Otherwise the output will be ``fail". In this example the
map is an isomorphism.
\end{example}

\begin{example}\label{girth}
Compute the diameter and girth of $\Gamma^{a'}_{34}(11)$.

\begin{lstlisting}
 Diameter[G34a']
 Girth[G34a']
\end{lstlisting}

In Example \ref{graphdraw}, after input the list into
\cite{Mathematica}, we just add the above two sentences. Then the
computer will show that the diameter is $10$ and the girth is $12$.
\end{example}

\subsection{Detailed classifications}

\begin{theorem}\label{detailed-classification}
(1) For each $g$, the MS graphs and their spatial graphs are:
\begin{itemize}

\item  $g = 11$.
Four MS graphs, three with a unique spatial graph:
$\Gamma_{28}^a({11})$, $\Gamma_{34}^a(11)$, $\Gamma_{34}^{a'}(11)$;
one with two spatial graphs $\Gamma_{28}^{a'}(11)$ and
$\Gamma_{34}^b(11,k)$;

\item $g = 241$.
Four MS graphs, each with a unique spatial graph: $\Gamma_{38}^a(241)$,
$\Gamma_{38}^{a'}(241)$, $\Gamma_{38}^b(241, k)$,
$\Gamma_{38}^c(241, k)$;

\item $g = 3,5,7,17,19,29,41,97,601, 1681$. For each genus, two MS graphs, each with a unique (unknotted) spatial graph. The two graphs come from pairs of dual edges in the orbifolds listed in Theorem \ref{position};

\item $g = 841$.
Two MS graphs, each with a unique spatial graph: $\Gamma_{22D}^a(841)$,
$\Gamma_{19}^a(841)$;

\item  $g = 9, 121, 361$. For each genus, one MS
graph, each with two spatial graphs:
$\Gamma_{20C}^a(9)$ and  $\Gamma_{20C}^b(9, k)$,
$\Gamma_{22B}^a(121)$ and  $\Gamma_{22B}^b(121, k)$,
$\Gamma_{22C}^a(361)$ and  $\Gamma_{22C}^b(361, k)$;

\item  $g = 21, 481$. For each genus, one MS graph, each with infinitely many different knotted spatial graphs: $\Gamma_{28}^d(21, k) \cong
\Gamma_{34}^c(21, k)$, $\Gamma_{38}^d(481, k)$;

\item For all other values of $g$, there is just one MS graph with a unique
(unknotted) spatial graph.
\end{itemize}

(2)  For each spatial MS graph $\Gamma$, there
is a unique MS action. This is also true for each (abstract) MS graph $\Gamma$ except $\Gamma_{28}^c(29)\cong\Gamma_{15E}^a(29)$,
where the actions correspond to
$\Gamma_{28}^c(29)$ and $\Gamma_{15E}^a(29)$ are different.
\end{theorem}

\begin{proof}
For each minimal graph $\Gamma$, denote the set of  invariants by
$$\Lambda(\Gamma)=\{ d_k(\Gamma),   E(\Gamma), D(\Gamma), \mathcal{G}(\Gamma)\},$$
where  $d_k$, $E$,  $D$  and $\mathcal{G}$ denote the number of
vertices of degree $k$ ($k=2,3,...$),
 the number of edges,  the diameter,  and the girth of $\Gamma$ respectively.

For a given $g$,  $\gamma$ in $\mathcal{O}_N$ realizing $m_g$ are
listed in Theorem 3.3, and $\Gamma_N^\gamma(g)$ are listed in
Theorem \ref{detailed-classification}. We will be able to picture
all those  graphs $\Gamma_N^\gamma(g)$ and to compute their
invariants  $\Lambda$ by the methods in last  subsection, which are
shown as in Appendix \ref{appendixA} (the pictures of $\Gamma$ with $E(\Gamma)>150$ are not given). As a result, graphs in the
discussion with the same invariants $\Lambda$ are listed as below:

(i) $\Lambda(\Gamma_{28}^c(29))=\Lambda(\Gamma_{15E}^a(29))$,

(ii) $\Lambda(\Gamma_{28}^{a'}(11))=\Lambda(\Gamma_{34}^b(11,k))$,

(iii) $\Lambda(\Gamma_{20C}^a(9))=\Lambda(\Gamma_{20C}^b(9, k))$,

(iv) $\Lambda(\Gamma_{22B}^a(121))=\Lambda(\Gamma_{22B}^b(121, k))$,

(v) $\Lambda(\Gamma_{22C}^a(361))=\Lambda(\Gamma_{22C}^b(361, k))$,

(vi) $\Lambda(\Gamma_{28}^d(21, k))=\Lambda(\Gamma_{34}^c(21, k))$ for any knotted dashed arcs $ d$ and $c$,

(vii) $\Lambda(\Gamma_{38}^d(481, k))$ is constant for all knotted dashed arcs $d$.

Therefore for any genus which does not appear in the list (i)-(vii), Theorem \ref{detailed-classification} is proved, and for the remaining genera, we only need to consider the graphs listed in (i)-(vii).

At what follows, we use $\mathbb{Z}_n$ to denote the cyclic group with order $n$, use $D_n$ to denote the dihedral group with order $2n$, use $S_n$ to denote the permutation group with order $n!$, and use $A_n$ to denote the alternating group with order $n!/2$.

It is clear that $\Gamma_{28}^c(29)$ and $\Gamma_{15E}^a(29)$ are the same as abstract graphs (both are isomorphic to $K_{(2, 30)}$, the complete bipartite graph with vertices partitioned into two subsets of cardinality 2 and 30),
but their corresponding group actions are different: one group is isomorphic to
$A_5\times
\mathbb{Z}_2$ and the other is isomorphic to $D_{30}\times \mathbb{Z}_2$ . There is no diffeomorphism sending $(S^3,  \Gamma_{28}^c(29))$ to $(S^3,  \Gamma_{15E}^a(29))$
since at the vertex of degree 30 in $(S^3, \Gamma_{15E}^a(29))$ all edges are tangent to some plane,
and this is not the situation for $(S^3, \Gamma_{28}^c(29))$.
So $\Gamma_{28}^c(29)$ and $\Gamma_{15E}^a(29)$ are not the same spatial graph.

Two graphs in each pair of  (ii), (iii), (iv),  (v) are different
spatial graphs, since one is knotted and the other is unknotted. We
will prove that these two graphs are $G$-equivalent by the method
discussed in last subsection, therefore finishing the proof of Theorem
\ref{detailed-classification} for genus 11, 9, 121, and 361.



When $g=11$, we will prove that the
representations induced by $a'$ in $\mathcal{O}_{28}$ and $b$ in
$\mathcal{O}_{34}$ are equivalent.
Using the Wirtinger presentation, we have the following.

\begin{figure}[h]
\centerline{\scalebox{0.8}{\includegraphics{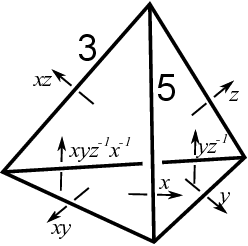}} \hspace{20pt}
\scalebox{0.6}{\includegraphics{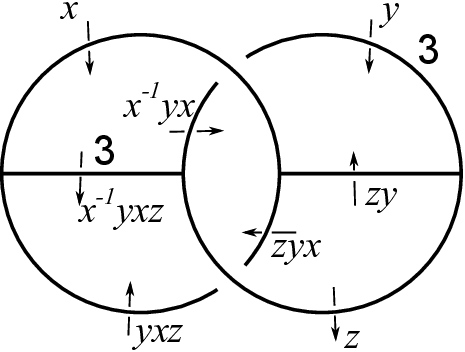}}} \caption{Presentations
of $\pi_1(\mathcal{O}_{28})$ and
$\pi_1(\mathcal{O}_{34})$}\label{fig:Pre2834}
\end{figure}
\begin{center}
\begin{tabular}{lcl}
$\pi_1(\mathcal{O}_{28})$&$=$&$\langle x,y,z\mid
x^5,y^2,z^2,(xz)^3,(xy)^2,(yz^{-1})^2\rangle$,\\
$\pi_1(\mathcal{O}_{34})$&$=$&$\langle x, y, z \mid
x^2,y^3,z^2,(zy)^2,(yxz)^2,(yxzx)^3\rangle$.
\end{tabular}
\end{center}

$\pi_1(a')\cong\pi_1(b)\cong D_2*_{\mathbb{Z}_2}D_3$. We choose
three generators for $\pi_1(a')$: $u$ is the generator of $D_2\cap
D_3\cong \Z_2$, $u_l$ is an order $2$ element in $D_2$ different
from $u$, $u_r$ is an order $3$ element in $D_3$. Similarly choose
generators $v$, $v_l$, $v_r$ for $\pi_1(b)$. Then the equivalence is
given by:

\begin{center}
\begin{tabular}{lcl}
$\pi_1(a')\rightarrow
\pi_1(\mathcal{O}_{28})$&$:$&$(u,u_l,u_r)\mapsto
(xyz^{-1}x^{-1}, xzx^{-1}, xz)$,\\
$\pi_1(b)\rightarrow \pi_1(\mathcal{O}_{34})$&$:$&$(v,v_l,v_r)
\mapsto (x^{-1}, y^{-1}z^{-1}, z^{-1}x^{-1}y^{-1}x)$,\\
$\pi_1(a')\rightarrow \pi_1(b)$&$:$&$(u,u_l,u_r)\mapsto (v,v_l,v_r)$,\\
$\pi_1(\mathcal{O}_{28})\rightarrow \pi_1(\mathcal{O}_{34})$&$:$&$
(\overline{u},\overline{u}_l,\overline{u}_r)\mapsto
(\overline{v},\overline{v}_l,\overline{v}_r)$.
\end{tabular}
\end{center}

Note that all the possible homomorphisms from a finitely presented group to a finite group are finitely many, and by using computer, for example \cite{GAP}, one can enumerate all such homomorphisms. Indeed we find the above homomorphisms by such a way.

When $g = 9, 121, 361$, we need to prove that for $N=20C, 22B, 22C$
the representations induced by $a$ and $b$ are equivalent.

$N=20C$:
\begin{figure}[h]
\centerline{\scalebox{0.8}{\includegraphics{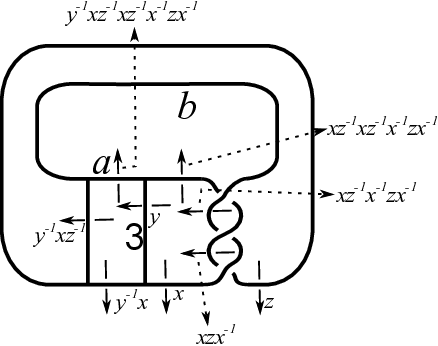}}}
\caption{Presentation of
$\pi_1(\mathcal{O}_{20C})$}\label{fig:Pre20C}
\end{figure}
\begin{center}
$\pi_1(\mathcal{O}_{20C})=\langle x, y, z\mid x^2, y^3, z^2,
(y^{-1}x)^2, (y^{-1}(xz)^3x)^2, (y^{-1}xz^{-1})^2\rangle$.
\end{center}

$\pi_1(a)\cong\pi_1(b)\cong D_2*_{\mathbb{Z}_2}D_3$. We choose three
generators for $\pi_1(a)$: $u$ is the generator of $D_2\cap D_3\cong
\mathbb{Z}_2$, $u_l$ is an order $2$ element in $D_2$ different from
$u$, $u_r$ is an order $3$ element in $D_3$.  Similarly choose
generators $v$, $v_l$, $v_r$ for $\pi_1(b)$. Then the equivalence is
given by:

\begin{center}
\begin{tabular}{lcl}
$\pi_1(a) \rightarrow \pi_1(\mathcal{O}_{20C})$ &$:$&
$(u,u_l,u_r)\mapsto (y^{-1}(xz)^3x, y^{-1}xz^{-1}, y)$,\\
$\pi_1(b) \rightarrow \pi_1(\mathcal{O}_{20C})$ &$:$&
$(v,v_l,v_r)\mapsto ((xz)^2x, y^{-1}(xz)^3x, xzxyxzx)$,\\
$\pi_1(a)\rightarrow \pi_1(b)$ &$:$& $(u,u_l,u_r)\mapsto
(v,v_l,v_r)$.
\end{tabular}
\end{center}

$N=22B$:
\begin{figure}[h]
\centerline{\scalebox{0.8}{\includegraphics{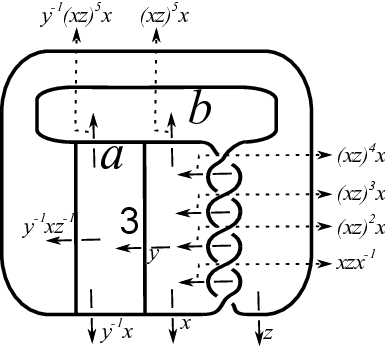}}}
\caption{Presentation of
$\pi_1(\mathcal{O}_{22B})$}\label{fig:Pre22B}
\end{figure}
\begin{center}
$\pi_1(\mathcal{O}_{22B})=\langle x, y, z\mid x^2, y^3, z^2,
(y^{-1}x)^2, (y^{-1}(xz)^5x)^2, (y^{-1}xz^{-1})^2\rangle.$
\end{center}

$\pi_1(a)\cong\pi_1(b)\cong D_2*_{\mathbb{Z}_2}D_3$. We choose three
generators for $\pi_1(a)$: $u$ is the generator of $D_2\cap D_3\cong
\mathbb{Z}_2$, $u_l$ is an order $2$ element in $D_2$ different from
$u$, $u_r$ is an order $3$ element in $D_3$.  Similarly choose
generators $v$, $v_l$, $v_r$ for $\pi_1(b)$. Then the equivalence is
given by:

\begin{center}
\begin{tabular}{lcl}
$\pi_1(a)\rightarrow \pi_1(\mathcal{O}_{22B})$ &$:$&
$(u,u_l,u_r)\mapsto (y^{-1}(xz)^5x, y^{-1}xz^{-1}, y)$,\\
$\pi_1(b)\rightarrow \pi_1(\mathcal{O}_{22B})$ &$:$&
$(v,v_l,v_r)\mapsto ((xz)^4x, y^{-1}xz^{-1}, (xzxz)y(xzxz)^{-1})$,\\
$\pi_1(a)\rightarrow \pi_1(b)$ &$:$& $(u,u_l,u_r)\mapsto
(v,v_l,v_r)$.
\end{tabular}
\end{center}

$N=22C$:
\begin{figure}[h]
\centerline{\scalebox{0.8}{\includegraphics{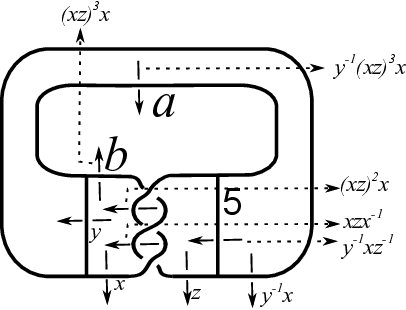}}}
\caption{Presentation of
$\pi_1(\mathcal{O}_{22C})$}\label{fig:Pre22C}
\end{figure}
\begin{center}
$\pi_1(\mathcal{O}_{22C})=\langle x, y, z\mid x^2, y^2, z^2,
(y^{-1}x)^2, (y^{-1}(xz)^3x)^2, (y^{-1}xz^{-1})^5\rangle.$
\end{center}

$\pi_1(a)\cong\pi_1(b)\cong D_2*_{\mathbb{Z}_2}D_5$.  We choose
three generators for $\pi_1(a)$: $u$ is the generator of $D_2\cap
D_5\cong \mathbb{Z}_2$, $u_l$ is an order $2$ element in $D_2$
different from $u$, $u_r$ is an order $5$ element in $D_5$.
Similarly choose generators $v$, $v_l$, $v_r$ for $\pi_1(b)$. Then
the equivalence is given by:

\begin{center}
\begin{tabular}{lcl}
$\pi_1(a)\rightarrow \pi_1(\mathcal{O}_{22C})$ &$:$&
$(u,u_l,u_r)\mapsto (x^{-1}y(zx)^{-3}, x^{-1}yx,
z^{-1}y^{-1}x)$,\\
$\pi_1(b)\rightarrow \pi_1(\mathcal{O}_{22C})$ &$:$&
$(v,v_l,v_r)\mapsto ((xz)^{-2}z^{-1},
x^{-1}y(zx)^{-3}, zxy(xz)^2)$,\\
$\pi_1(a)\rightarrow \pi_1(b)$ &$:$& $(u,u_l,u_r)\mapsto
(v,v_l*v,v_r^2)$.
\end{tabular}
\end{center}

In the last two cases, (vi) and(vii), that is, $g = 21$ and $481$, we need to prove that for $N=28, 34$ and for
$N=38$ all the representations induced by allowable dashed arcs are
equivalent.

Notice that in these cases $\Gamma/G$ are all dashed arcs whose
fundamental groups are isomorphic to $\Z_2\ast\Z_3$. By \cite{Du2},
the fundamental groups of orbifolds $\mathcal{O}_{28}$,
$\mathcal{O}_{34}$ and $\mathcal{O}_{38}$ are isomorphic to $I\times
\mathbb{Z}_2$, $A_5\times \mathbb{Z}_2$ and $\mathbf{I\times O}$
separately. The notations are explained as follows.

Let $O$ be the orientation-preserving isometry group of a regular octahedron (or a cube), and let $I$ be
the orientation-preserving isometry group of a regular icosahedron (or a regular dodecahedron). Then clearly $O$ and $I$ belong to $SO(3)\subset SO(4)$. Hence they also act on $S^3$. $\mathbf{I\times O}$ is the two-sheeted covering of $I\times O$ under the two-to-one map $SO(4)\rightarrow SO(3)\times SO(3)$, and it acts on $S^3$. $A_5$ acts on a regular 4-simplex, which has 5 vertices, as the orientation-preserving isometry group. Hence it also acts on $S^3$. Finally the $\mathbb{Z}_2$ summand acts on $S^3$ as an antipodal map which commutes with the action given by $I$ or $A_5$. Note that as abstract groups $O\cong S_4$ and $I\cong A_5$. However, the actions given by $I$ and $A_5$ are different.

Then the results can be derived from Lemma \ref{same} below.
\end{proof}

\begin{remark}
In \cite{Du2}, $J$ is used to denote $I$ and notations
$\mathbf{J\times_J J}$, $\mathbf{J\times^{*}_J J}$ and
$\mathbf{J\times O}$ are used to denote the fundamental groups of
$\mathcal{O}_{28}$, $\mathcal{O}_{34}$ and $\mathcal{O}_{38}$.
\end{remark}

\begin{lemma}\label{same}
Suppose that $G$ is either $A_5\times
\mathbb{Z}_2$ or $\mathbf{I\times O}$, and that $\{x, y\}$ and $\{z, w\}$
both generate $G$, where $x, z$ both have order $2$ and $y, w$ both have order $3$. Then the map $x\mapsto z, y\mapsto w$ gives an automorphism of $G$.
\end{lemma}

\begin{proof}
Note that $I\cong A_5$ and $O\cong S_4$, hence $\mathbf{I\times O}$ is
a two-sheeted covering of $A_5\times S_4$. Using the permutation
representation, we may assume that  $A_5$ acts on $\{1,2,3,4,5\}$,
$\mathbb{Z}_2$ acts on $\{6,7\}$, and $S_4$ acts on $\{6,7,8,9\}$.

For $A_5\times\Z_2$, if an
order $2$ element and an order $3$ element generate the group, then
the two generators have forms like $(12)(34)(67)$ and $(135)$. It
is not hard to see that the map between two such generating sets is an
isomorphism.

For $A_5\times S_4$,  if an
order $2$ element and an order $3$ element generate the group, then
the two generators have forms like $(12)(34)(69)$ and $(135)(678)$.
Also the map between two such generating sets is an isomorphism, and
$A_5\times S_4$ has the property stated in the lemma.

Now if an order $2$ element $x$ and an order $3$ element $y$
generate $\mathbf{I\times O}$, then the quotients $\overline{x}$ and
$\overline{y}$ in $A_5\times S_4$ have the above forms.
$\overline{y}$ has two preimages, one is $y$, the other has order
$6$. Hence all such $y$'s are conjugate in $\mathbf{I\times O}$.
Hence if $\{z, w\}$ satisfies the condition in the lemma, we can
assume that $w=y$. Then considering the permutation representation and
the two-to-one covering, the possible $z$ has $36$ choices.

Then using \cite{GAP}, one can confirm that all the $36$ choices
give equivalent generator pairs.
\end{proof}








\section{Maximum order of orientation-reversing case}

Now we can consider maximum order problems for extendable actions on
handlebodies and graphs for the general case, which allows group elements which reverse orientations.

Suppose $G$ acts on $(S^3, U_\epsilon(\Gamma))$. Then for any $g\in
G$, $g$ will preserve or reverse the orientations of
$U_\epsilon(\Gamma)$ and $S^3$ simultaneously, therefore preserve or
reverse the orientations of $\Sigma_g=\partial U_\epsilon(\Gamma)$ and $S^3$ simultaneously. That is to say, in
graph and handlebody cases if there is an orientation-reversing
element in $G$, $G$ must be of type $(-,-)$ on $(S^3, \partial U_\epsilon(\Gamma))$ as defined in
\cite{WWZ}. Then the maximum order of extendable actions for
orientation-reversing case will be $E_g(-,-)$, which is presented as
Proposition 4.7 in \cite{WWZ}.
By definition, the action $G$ realizing $E_g(-,-)$ is faithful on $\partial U_\epsilon(\Gamma)$,
therefore is faithful on the
handlebody $U_\epsilon(\Gamma)$, but may not be faithful on the graph $\Gamma$.

\begin{proposition}\label{[WWZ]} (Proposition 4.7 of \cite{WWZ} and its proof, also refer proof of Theorem \ref{max-OR} for the ``Moreover'' part)
$E_g(-,-)$ are given as below: {\small\begin{center}
\begin{tabular}{|c|c|}
\hline \raisebox{-1pt}{$E_g(-,-)$} & $g$\\\hline $24(g-1)$ & $3, 5,
6, 11, 17, 97, 601$\\\hline $16(g-1)$ & $7, 73$\\\hline $40(g-1)/3$
& $19$\\\hline $48(g-1)/5$ & $41$\\\hline $60(g-1)/7$ &
$1681$\\\hline $8(\sqrt{g}+1)^2$ & $k^2$, $k>1$\\\hline $8(g+1)$ &
the remaining numbers\\\hline
\end{tabular}
\end{center}}

Moreover  for each $g$,  $E_g(-,-)$ is realized (in the  orbifold
level) by a singular edge/dashed arc $\gamma$ of a 3-orbifold
$\mathcal{O}_N$ shown as (a) or (b) in Figure 8, where a
``reflection'' of $\mathcal{O}_N$ fixes $\gamma$ (in (a), the fixed-point set of the reflection is the plane of the paper and in (b), the fixed-point set of the reflection is the indicated plane). We omit the labels
here for convenience.
\begin{figure}[h]
\centerline{\scalebox{0.6}{\includegraphics{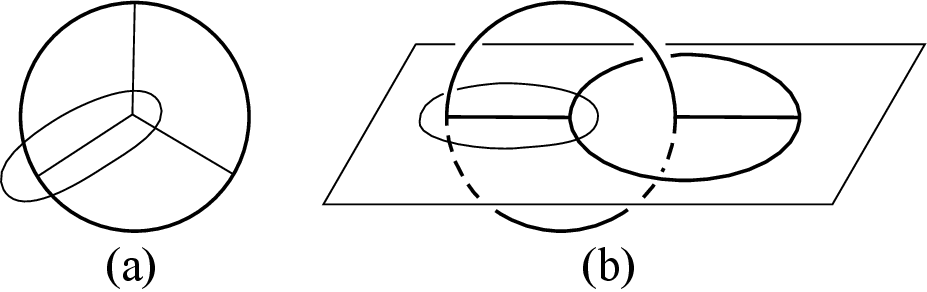}}}
\caption{Orbifolds with reflections}\label{fig:OadmitR}
\end{figure}
\end{proposition}

Similar to Proposition  \ref{SO(4)}, we have the following.
\begin{proposition}\label{O(4)}
Suppose $G\subset O(4)$ acts on $(S^3, \Gamma)$, $g\in G$  is not
the identity of $G$ and its action  on $\Gamma$ is the identity.
Then $g$ must be a reflection about a geodesic 2-sphere $S^2$ in
$S^3$, and $\Gamma\subset S^2$.
\end{proposition}

\begin{proof}
By Proposition \ref{SO(4)}, we may assume that $g$ is orientation-reversing. By results in linear algebra, $\text{Fix}(g)$ is
either a pair of points or a geodesic 2-sphere
$S^2\subset S^3$. Since $\Gamma\subset  \text{Fix}(g)$,
$\text{Fix}(g)$ must be a geodesic $S^2(\supset\Gamma)$ and $g$
interchanges the two 3-balls separated by $S^2$.
\end{proof}


Suppose that the $G$-action on $\Gamma$ is not faithful. Then it is easy
to see that if we $G$-equivariantly add some free edges to $\Gamma$
perpendicular to the geodesic $S^2$ containing $\Gamma$ to get
$\Gamma^*$, then $G$ acts faithfully on $\Gamma^*$.

Now let  $E(V_g)$, $M^*_g$ and $M_g$ be the general maximum orders of extendable group actions on
handlebodies, arbitrary graphs, and minimal graphs,  of genus $g$ respectively. Then we have

\begin{theorem}\label{max-OR}  (1) $m_g=OE_g\le M_g \le E(V_g)=M^*_g.$

(2) $E(V_g)=M^*_g$  are given in the following table.

\begin{center}
\begin{tabular}{|c|c|}
\hline $E(V_g)=M^*_g$ & $g$\\\hline $24(g-1)$ & $2, 3, 4, 5, 6, 11, 17,
97, 601$\\\hline $16(g-1)$ & $7, 9, 73$\\\hline $40(g-1)/3$ & $16,
19$\\\hline $12(g-1)$ & $25, 121, 241$\\\hline $48(g-1)/5$ &
$41$\\\hline $60(g-1)/7$ & $29, 841, 1681$\\\hline $8(\sqrt{g}+1)^2$
& $k^2$, $k\neq 11$\\\hline $8(g+1)$ & remaining numbers\\\hline
\end{tabular}
\end{center}

(3) $M_g$  are given in the following table.

\begin{center}
\begin{tabular}{|c|c|}
\hline $M_g$ & $g$\\\hline $24(g-1)$ & $3, 4, 5, 6, 11, 17,
97, 601$\\\hline $16(g-1)$ & $7, 9, 73$\\\hline $40(g-1)/3$ & $16,
19$\\\hline $12(g-1)$ & $2, 25, 121, 241$\\\hline $48(g-1)/5$ &
$41$\\\hline $60(g-1)/7$ & $29, 841, 1681$\\\hline $8(\sqrt{g}+1)^2$
& $k^2$, $k\neq 11$\\\hline $8(g+1)> M_g \ge 4(g+1)$ & remaining numbers\\\hline
\end{tabular}
\end{center}
\end{theorem}

\begin{proof}
(1) follows from the definitions, Theorem \ref{mg} and the paragraph after Proposition \ref{O(4)}.

(2) Suppose that $G$ acts on $V_g$ realizing $E(V_g)$. Then $G$ is either
orientation-preserving or of the type $(-,-)$. To get the table in
(2), we need only  to compare tables in  Theorem \ref{mg} (see also
Remark 3.4) and in Proposition \ref{[WWZ]}, where $E(V_g)$ are
chosen from the former for $g=25,121, 241$, and  from the later for
the remaining $g$. Note that for table in (2) we can put 4, 9, 16, 25,
841  into the line of $k^2$, and 29 into the bottom line, then it
has the form  in the table of Proposition \ref{[WWZ]}.

(3) By (2) and its proof, there is nothing to be further verified for $g=25,121, 241$, and to prove (3),
we need to verify the following

{\bf Claim} for all remaining $g$, $M_g=E_g(-,-)$ if and only if $g$
is not in the bottom line and $g\ne 2$, and further more $M_2=12$.

(*) Note first that the ``Moreover'' part of Proposition \ref{[WWZ]}
can be interpreted as that, for each $g$, $E_g(-,-)$ is realized by
the orientation-preserving maximum symmetry of $\Gamma_N^\gamma(g)$
and reflections about geodesic 2-spheres which keep
$\Gamma_N^\gamma(g)$ invariant. Therefore if  $\Gamma_N^\gamma(g)$
does not stay in any geodesic 2-sphere, then the group action
realizing $E_g(-,-)$ acts faithfully on $\Gamma_N^\gamma(g)$.

(**) Note then that $\Gamma_N^\gamma(g) \subset S^3$ does not stay
in any geodesic 2-sphere in $S^3$ if a vertex $v$ of  $\gamma$ in
the singular set of $\mathcal{O}_N$ has three adjacent edges with
index $(2, 3, q)(q=3,4,5)$, since for any vertex $v'$ of
$\Gamma_N^\gamma(g)$ in $p^{-1}(v)$, where $p: S^3\to \mathcal{O}_N$
is the orbifold covering, the action of the stabilizer $St(v')$ on a
neighborhood of $v'$ is the same as the isometry of some regular
polyhedron, and adjacent edges of $v'$ can not lie in any geodesic
two sphere.


Now we are going to verify the ``if'' part of the claim by  finding
$\Gamma_N^\gamma(g)$ realizing $E_g(-,-)$ so that  either $\gamma$
in $  \mathcal{O}_N$ meets (**), or more directly
$\Gamma_N^\gamma(g)$ is non-planar. These will be carried in (i) and
(ii) below respectively. Therefore the group action $G$ realizing
$E_g(-,-)$ acts faithfully on $\Gamma_N^\gamma(g)$.

(i) For $g=3, 5, 11, 6, 17, 97, 601, 7, 73, 19,  41, 1681, 29$, (we follow the order of their appearance in Theorem 3.3), we choose
$\Gamma_N^\gamma(g)$ to be $\Gamma_{26}^{a}(3)$,  $\Gamma_{27}^{a}(5)$,  $\Gamma_{28}^{a}11)$, $\Gamma_{24}^{a}(6)$,
$\Gamma_{29}^{a}(17)$, $\Gamma_{33}^{a'}(97)$, $\Gamma_{30}^{a'}(601)$, $\Gamma_{27}^{b}(7)$, $\Gamma_{25}^{a}(73)$,
$\Gamma_{28}^{b}(19)$, $\Gamma_{29}^{b'}(41)$, $\Gamma_{30}^{b'}(1681)$, $\Gamma_{28}^{c'}(29)$.

(ii) For all the  squares $g=k^2$, we choose
$\Gamma_N^\gamma(g)$ to be $\Gamma_{19}^{a}(k^2)$, where the parameter $n$  in $\mathcal{O}_{19}$
is chosen to be $n=k+1$. As an abstract graph $\Gamma_{19}^{a}(k^2)$ can be obtained from the complete bipartite graph $K_{(k+1, k+1)}$ by adding vertices. $K_{(k+1, k+1)}$ is a well-known non-planar graph, hence $\Gamma_{19}^{a}(k^2)$ can not
lie in a geodesic sphere.

For $g=2$ and the remaining cases, the graph realizing $E_g(-,-)$ is
$\Gamma_{15E}^a(g)$, which is abstractly isomorphic to
$K_{(2,g+1)}$, indeed lies in a geodesic 2-sphere,  and the action
on $\Gamma_{15E}^a(g)$ realizing $E_g(-,-)$ is not faithful (see
also Appendix \ref{appendixB}).

By Theorem \ref{classify} and the proof of Theorem \ref{mg}, the
second biggest order of orientation-preserving extendable action on
a genus 2 graph is $6$ (corresponds to `$a$' in $\mathcal{O}_{01}$).
Hence $M_2\leq 12$, and we have $M_2=12$.
\end{proof}

\begin{example} Suppose $g=pq$, $p$ is the smallest non trivial divisor of $g$.
As the construction for $g=k^2$ for general case we can get a complete bipartite graph $K_{(p+1,q+1)}\subset S^3$ which has genus $g$ and on it there is an
extendable group action with order $4(p+1)(q+1)$ which is bigger
than $4(g+1)$ (see examples in \cite{WWZZ1} and \cite{WWZZ}). Clearly $K_{(p+1,q+1)}$ is non-planar.
\end{example}

\begin{conjecture}
Suppose $g$ is neither a square number  nor one of those finitely
many $g$ listed in the table above. Then $M_g$ is $4(g+1)$ for prime
$g$, and $4(p+1)(q+1)$ otherwise, where $pq=g$, $p$ is the smallest
non trivial divisor of $g$.
\end{conjecture}

\begin{remark}
(1) The maximum order of finite group action on minimal graphs of genus $g$ is $2^gg!$ if $g>2$ and is 12 if $g=2$ \cite{WZ}.

(2) Any faithful action of finite group $G$ on a minimal graph $\Gamma$ of genus $g$ provides an embedding
of $G$ into the out-automorphism group of the free group of rank $g$ for $g>1$.
\end{remark}

\newpage

\begin{appendix}
\section{Table of  MS graphs with invariants}\label{appendixA}
This appendix contains some basic invariants of all the MS graphs.

In the
following table, we denote by $d_k$ the number of vertices of degree
$k$; $E$ the number of edges of a graph; $D$ the diameter of a
graph; $\mathcal{G}$ the girth (or the length of a minimal loop) of
a graph. Below we explain the last column  which presents more standard names of those graphs if they have.

We use $K_n$ to denote the complete graph with $n$ vertices and $K_{m,n}$ to denote the complete bipartite graph with $m+n$ vertices.
The generalized Petersen graph $G(n,k)$ is a graph with vertex set
$$\{u_0, u_1, \cdots, u_{n-1}, v_0, v_1, \cdots, v_{n-1}\}$$
and edge set
$$\{u_iu_{i+1}, u_iv_i, v_iv_{i+k}\mid i=0, \cdots, n-1\}$$
where subscripts are to be read modulo $n$ and $k < n/2$.

In the following table, we use $\widetilde{K}_n$ (resp. $\widetilde{K}_{m,n}$, $\widetilde{G}(n,k)$) to denote the graph obtained by adding one degree $2$ vertex at the middle of each edge of $K_n$ (resp. $K_{m,n}$, $G(n,k)$).
Moreover the precise meaning of  ``1-skeleton" in the last column also means a graph obtained by adding one degree $2$ vertex at the middle of each edge of that 1-skeleton.

\begin{adjustwidth}{-4em}{0em}
\begin{center}
\begin{tabular}{|l|ll|r|r|r|r|}
\hline \multicolumn{1}{|c|}{name} & \multicolumn{2}{|c|}{number of vertices} & \multicolumn{1}{|c|}{E} & \multicolumn{1}{|c|}{D} & \multicolumn{1}{|c|}{$\mathcal{G}$}& \multicolumn{1}{|c|}{annotation} \\
\hline \multicolumn{7}{|c|}{$|G|=12(g-1)$} \\
\hline $\Gamma_{15E}^a(2)$ & $d_2=3$ & $d_3=2$ & 6 & 2 & 4  & $K_{2,3}$\\
\hline $\Gamma_{26}^a(3)$ & $d_2=4$ & $d_4=2$ & 8 & 2 & 4 &  $K_{2,4}$\\
\hline $\Gamma_{26}^{a'}(3)$ & $d_2=6$ & $d_3=4$ & 12 & 4 & 6 & $1$-skeleton of  tetrahedron, $\widetilde{K}_4$\\
\hline $\Gamma_{19}^a(4)$ & $d_2=9$ & $d_3=6$ & 18 & 4 & 8  & $\widetilde{K}_{3,3}$\\
\hline $\Gamma_{27}^{a}(5)$ & $d_2=6$ & $d_6=2$ & 12 & 2 & 4 &$K_{2,6}$\\
\hline $\Gamma_{27}^{a'}(5)$ & $d_2=12$ & $d_3=8$ & 24 & 6 & 8 & $1$-skeleton of cube, $\widetilde{G}(4,1)$\\
\hline $\Gamma_{24}^{a}(6)$ & $d_2=10$ & $d_4=5$ & 20 & 4 & 6 & $1$-skeleton of $4$-simplex, $\widetilde{K}_5$\\
\hline $\Gamma_{20C}^{a}(9)$ & \multirow{2}{*}{$d_2=24$} & \multirow{2}{*}{$d_3=16$} & \multirow{2}{*}{48} & \multirow{2}{*}{8} & \multirow{2}{*}{12} &  M\"{o}bius-Kantor graph,\\
       $\Gamma_{20C}^{b}(9, k)$ & & & & & &$\widetilde{G}(8,3)$\\
\hline $\Gamma_{28}^{a}(11)$ & $d_2=12$ & $d_{12}=2$ & 24 & 2& 4 &$K_{2,12}$\\
\hline $\Gamma_{28}^{a'}(11)$ & \multirow{2}{*}{$d_2=30$} & \multirow{2}{*}{$d_3=20$} & \multirow{2}{*}{60} & \multirow{2}{*}{10} & \multirow{2}{*}{10} & $1$-skeleton of dodecahedron,  \\
       $\Gamma_{34}^{b}(11, k)$ & & & & & & $\widetilde{G}(10,2)$\\
\hline $\Gamma_{34}^{a}(11)$  & $d_2=20$ & $d_4=10$ & 40 & 6 & 8 &\\
\hline $\Gamma_{34}^{a'}(11)$  & $d_2=30$ & $d_3=20$ & 60 & 10& 12  & $\widetilde{G}(10, 3)$, Desargues graph\\
\hline $\Gamma_{29}^{a}(17)$  & $d_2=24$ & $d_6=8$ & 48 & 4& 6&$1$-skeleton of  $16$-cell\\
\hline $\Gamma_{29}^{a'}(17)$  & $d_2=32$ & $d_4=16$ & 64 & 8& 8& $1$-skeleton of $4$-dim cube\\
\hline $\Gamma_{21A}^{a}(25)$  & $d_2=72$ & $d_3=48$ & 144 & 12& 16&\\
\hline $\Gamma_{33}^{a}(97)$  & $d_2=288$ & $d_3=192$ & 576 & 24& 16 & \\
\hline $\Gamma_{33}^{a'}(97)$  & $d_2=144$ & $d_6=48$ & 288 & 8& 8 & \\
\hline $\Gamma_{22B}^{a}(121)$ & \multirow{2}{*}{$d_2=360$} & \multirow{2}{*}{$d_3=240$} & \multirow{2}{*}{720} & \multirow{2}{*}{22} & \multirow{2}{*}{20}  & \\
       $\Gamma_{22B}^{b}(121,k)$ & & & & & &\\
\hline $\Gamma_{38}^{a}(241)$  & $d_2=480$ & $d_4=240$ & 960 & 20& 12&\\
\hline $\Gamma_{38}^{b}(241,k)$  & $d_2=720$ & $d_3=480$ & 1440 & 21& 24& \\
\hline $\Gamma_{38}^{c}(241,k)$  & $d_2=480$ & $d_4=240$ & 960 & 14& 16& \\
\hline $\Gamma_{38}^{a'}(241)$  & $d_2=720$ & $d_3=480$ & 1440 & 30& 18& \\
\hline $\Gamma_{30}^{a}(601)$  & $d_2=720$ & $d_{12}=120$ & 1440 & 10& 6& $1$-skeleton of  $600$-cell\\
\hline $\Gamma_{30}^{a'}(601)$  & $d_2=1200$ & $d_{4}=600$ & 2400 & 30& 10& $1$-skeleton of $120$-cell\\
\hline \multicolumn{7}{|c|}{$|G|=8(g-1)$} \\
\hline $\Gamma_{27}^{b}(7)$ & $d_2=8$ & $d_8=2$ & 16 & 2 & 4 & $K_{2,8}$\\
\hline $\Gamma_{27}^{b'}(7)$  & $d_2=12$ & $d_4=6$ & 24 & 4& 6& $1$-skeleton of  octahedron\\
\hline $\Gamma_{21B}^{a}(49)$  & $d_2=96$ & $d_4=48$ & 192 & 10 & 12&\\
\hline $\Gamma_{25}^{a}(73)$  & $d_2=96$ & $d_8=24$ & 192 & 6& 6& $1$-skeleton of  $24$-cell\\
\hline \multicolumn{7}{|c|}{$|G|=20(g-1)/3$} \\
\hline $\Gamma_{19}^{a}(16)$  & $d_2=25$ & $d_5=10$ & 50 & 4 & 8 & $\widetilde{K}_{5,5}$\\
\hline $\Gamma_{28}^{b}(19)$  & $d_2=20$ & $d_{20}=2$ & 40 & 2& 4&$K_{2,20}$\\
\hline $\Gamma_{28}^{b'}(19)$  & $d_2=30$ & $d_{5}=12$ & 60 & 6& 6& $1$-skeleton of  icosahedron\\
\hline $\Gamma_{22C}^{a}(361)$ & \multirow{2}{*}{$d_2=600$} & \multirow{2}{*}{$d_{5}=240$} & \multirow{2}{*}{1200} & \multirow{2}{*}{14} & \multirow{2}{*}{12} &\\
       $\Gamma_{22C}^{b}(361,k)$ & & & & & &\\
\hline \multicolumn{7}{|c|}{$|G|=6(g-1)$} \\
\hline $\Gamma_{28}^{d}(21,k)$ & \multirow{2}{*}{$d_2=960$} & \multirow{2}{*}{$d_3=40$} & \multirow{2}{*}{120} & \multirow{2}{*}{12} & \multirow{2}{*}{16} & \\
       $\Gamma_{34}^{c}(21,k)$ & & & & & & \\
\hline $\Gamma_{38}^{d}(481,k)$ & $d_2=1440$ & $d_3=960$ & 2880 & 36 & 24 &\\
\hline \multicolumn{7}{|c|}{$|G|=24(g-1)/5$} \\
\hline $\Gamma_{29}^{b}(41)$ & $d_3=32$ & $d_4=24$ & 96 & 6 & 6 & \\
\hline $\Gamma_{29}^{b'}(41)$ & $d_4=16$ & $d_8=8$ & 64 & 4 & 4 & \\
\hline \multicolumn{7}{|c|}{$|G|=30(g-1)/7$} \\
\hline $\Gamma_{30}^{b}(1681)$ & $d_3=1200$ & $d_5=720$ & 3600 & 20 & 6 & \\
\hline $\Gamma_{30}^{b'}(1681)$ & $d_4=600$ & $d_{20}=120$ & 2400 & 10 & 4 & \\
\hline
\end{tabular}
\end{center}
\end{adjustwidth}

\begin{adjustwidth}{-4em}{0em}
\begin{center}
\begin{tabular}{|l|ll|r|r|r|r|}
\hline \multicolumn{1}{|c|}{name} & \multicolumn{2}{|c|}{number of vertices} & \multicolumn{1}{|c|}{E} & \multicolumn{1}{|c|}{D} & \multicolumn{1}{|c|}{$\mathcal{G}$} & \multicolumn{1}{|c|}{annotation}\\
\hline \multicolumn{7}{|c|}{$|G|=4(\sqrt{g}+1)^2$ for $g=k^2, k\neq3, 5, 7, 11, 19, 41$} \\
\hline $\Gamma_{22D}^{a}(841)$ & $d_3=600$ & $d_5=360$ & 1800 & 12 & 8 &\\
\hline $\Gamma_{19}^{a}(841)$ & $d_2=900$ & $d_{30}=60$ & 1800 & 4 & 8 &$\widetilde{K}_{30, 30}$\\
\hline $\Gamma_{19}^{a}(k^2)$ & $d_2=(k+1)^2$ & $d_{k+1}=2(k+1)$ & $2(k+1)^2$ & 4 & 8 &$\widetilde{K}_{k+1, k+1}$\\

\hline \multicolumn{7}{|c|}{$|G|=4(g+1)$ for remaining $g$} \\
\hline $\Gamma_{28}^{c}(29)$ & \multirow{2}{*}{$d_2=30$} & \multirow{2}{*}{$d_{30}=2$} & \multirow{2}{*}{60} & \multirow{2}{*}{2} & \multirow{2}{*}{4} & \multirow{2}{*}{$K_{2,30}$}\\
       $\Gamma_{15E}^{a}(29)$ & & & & & &\\
\hline $\Gamma_{28}^{c'}(29)$ & $d_3=20$ & $d_{5}=12$ & 60 & 6 & 4 &\\

\hline $\Gamma_{15E}^{a}(g)$ & $d_2=g+1$ & $d_{g+1}=2$ & $2g+2$ & 2 & 4 & $K_{2,g+1}$\\
\hline
\end{tabular}
\end{center}
\end{adjustwidth}


Even if our computer program can produce any MS graph above, it is hard to understand the picture when the number of its edges is very large.
Following are the pictures of abstract MS graphs with less than $150$ edges and not   of type $K_{2,n}$ (there are many figures of this type which are very easy to understand).



\begin{center}
\parbox[t]{110pt}{\scalebox{0.5}{\includegraphics*{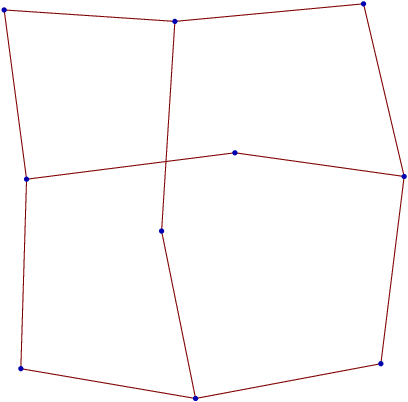}} \centerline{$\Gamma_{26}^{a'}(3)$}}
\parbox[t]{110pt}{\scalebox{0.5}{\includegraphics*{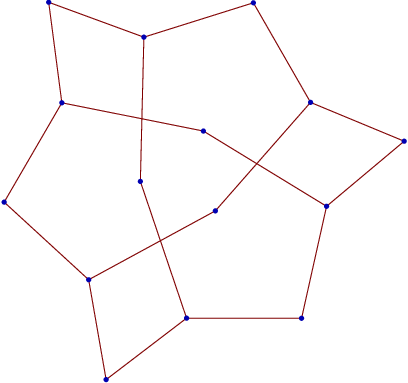}}  \centerline{$\Gamma_{19}^{a}(4)$}}
\parbox[t]{110pt}{\scalebox{0.5}{\includegraphics*{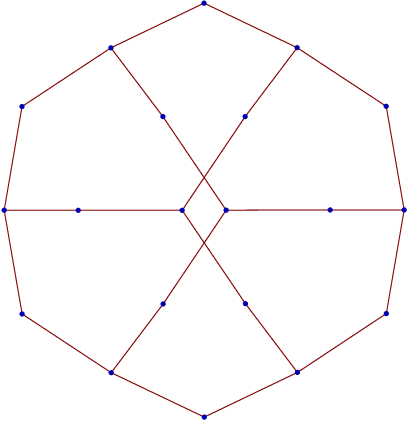}} \centerline{$\Gamma_{27}^{a'}(5)$}}
\end{center}

\begin{center}
\parbox[t]{110pt}{\scalebox{0.5}{\includegraphics*{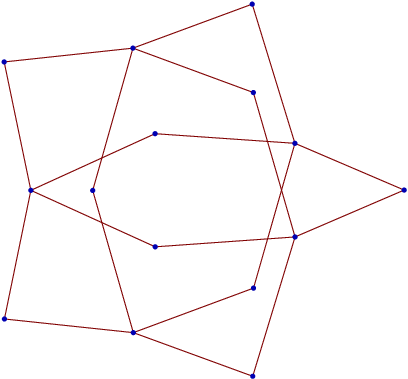}}  \centerline{$\Gamma_{24}^{a}(6)$}}
\parbox[t]{110pt}{\scalebox{0.5}{\includegraphics*{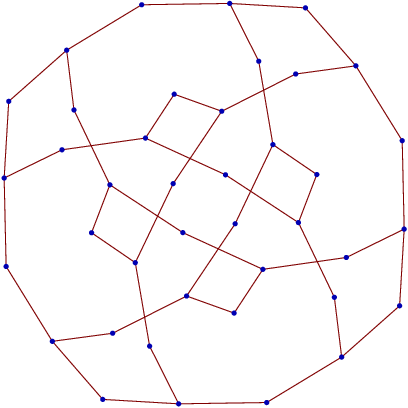}}  \centerline{$\Gamma_{20C}^{a}(9)\cong\Gamma_{20C}^{b}(9,k)$}}
\parbox[t]{110pt}{\scalebox{0.5}{\includegraphics*{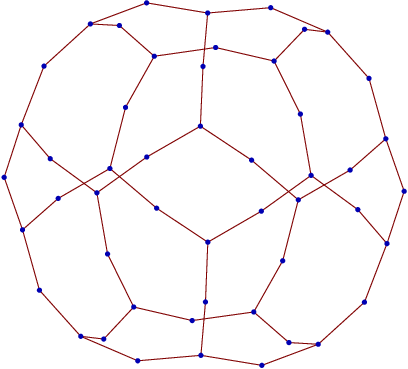}}  \centerline{$\Gamma_{28}^{a'}(11)\cong\Gamma_{34}^{b}(11,k)$}}
\end{center}

\begin{center}
\parbox[t]{110pt}{\scalebox{0.5}{\includegraphics*{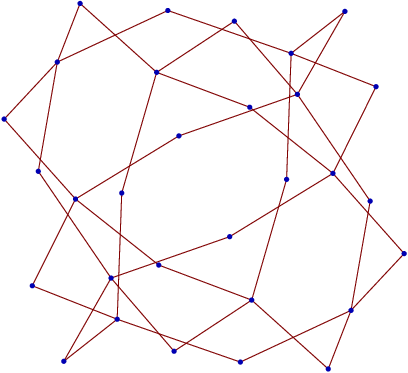}} \centerline{$\Gamma_{34}^{a}(11)$}}
\parbox[t]{110pt}{\scalebox{0.5}{\includegraphics*{Gr11d.eps}} \centerline{$\Gamma_{34}^{a'}(11)$}}
\parbox[t]{110pt}{\scalebox{0.5}{\includegraphics*{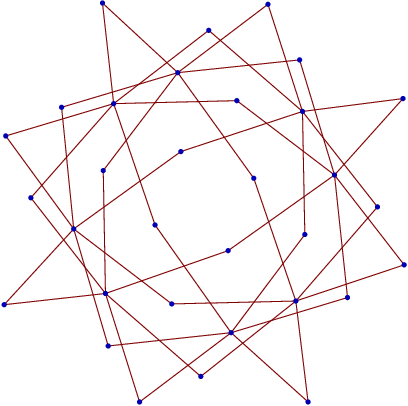}}  \centerline{$\Gamma_{29}^{a}(17)$}}
\end{center}

\begin{center}
\parbox[t]{110pt}{\scalebox{0.5}{\includegraphics*{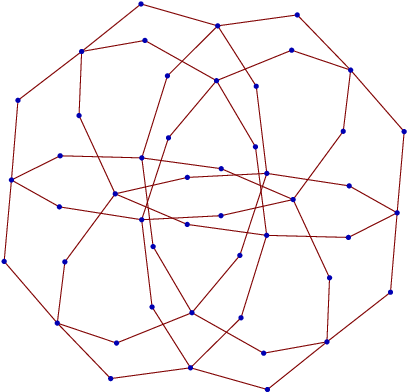}} \centerline{$\Gamma_{29}^{a'}(17)$}}
\parbox[t]{110pt}{\scalebox{0.5}{\includegraphics*{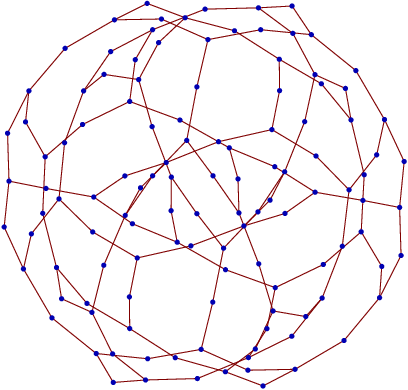}} \centerline{$\Gamma_{21A}^{a}(25)$}}
\parbox[t]{110pt}{\scalebox{0.5}{\includegraphics*{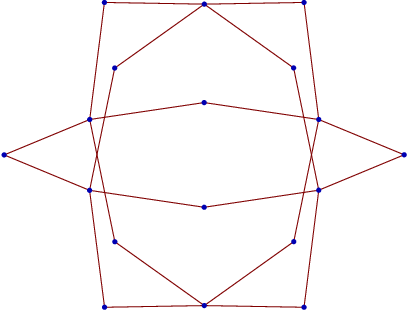}} \centerline{$\Gamma_{27}^{b'}(7)$}}
\end{center}

\begin{center}
\parbox[t]{110pt}{\scalebox{0.5}{\includegraphics*{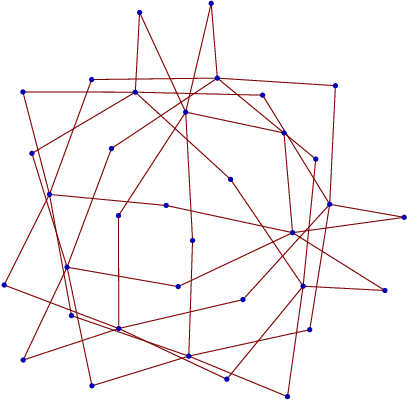}} \centerline{$\Gamma_{19}^{a}(16)$}}
\parbox[t]{110pt}{\scalebox{0.5}{\includegraphics*{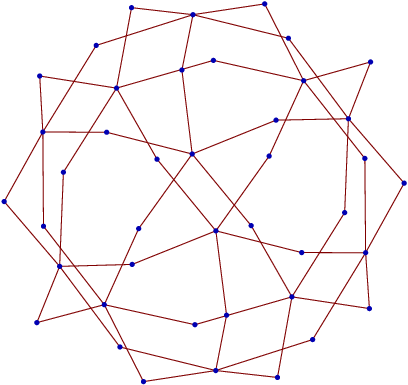}} \centerline{$\Gamma_{28}^{b'}(19)$}}
\parbox[t]{110pt}{\scalebox{0.5}{\includegraphics*{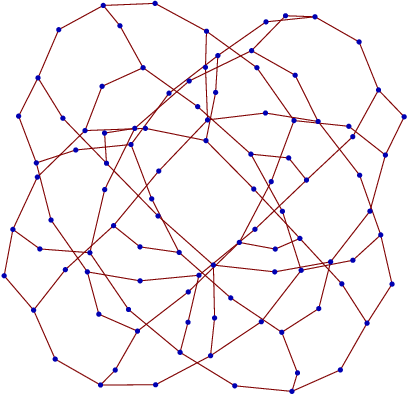}} \centerline{$\Gamma_{28}^{d}(21,k)$}}
\end{center}

\begin{center}
\parbox[t]{110pt}{\scalebox{0.5}{\includegraphics*{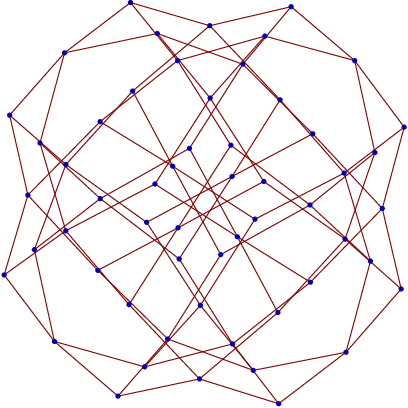}}  \centerline{$\Gamma_{29}^{b}(41)$}}
\parbox[t]{110pt}{\scalebox{0.5}{\includegraphics*{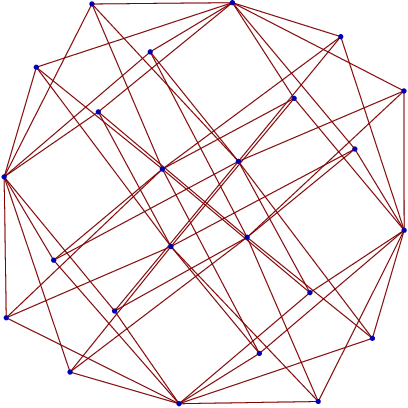}} \centerline{$\Gamma_{29}^{b'}(41)$}}
\parbox[t]{110pt}{\scalebox{0.5}{\includegraphics*{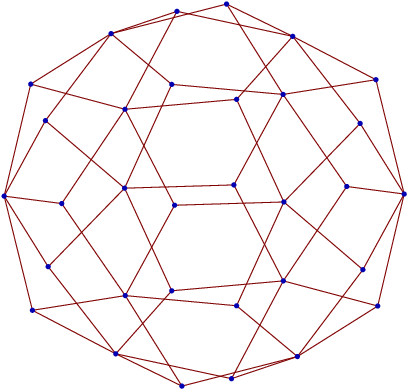}}  \centerline{$\Gamma_{28}^{c'}(29)$}}
\end{center}

\section{Description of spatial MS graphs}\label{appendixB}

With one point removed, the unit
sphere $S^3\subset E^4$ will be mapped to $\mathbb{R}^3$ by
a stereographic projection. Each MS spacial graph in $S^3$ consists of geodesic segments
as edges, which are mapped to circles in $\mathbb{R}^3$.

This appendix contains the pictures in $\mathbb{R}^3$ which are stereographic projections of some spatial MS graphs in $S^3$.
Those pictures present the geometrical shapes of those spatial MS graphs, and possibly also some beauty and complexity provided by geometry and computer programs.

For intuitive descriptions of the corresponding symmetries for many cases, please see also the last sections of \cite{WWZZ1} and \cite{WWZZ}.


Below we present the shape of those graphs: Two infinite series, some spacial graphs related to 3- and 4-dimensional regular polyhedra (2- and 3-dimensional spherical tessellation), finally one  example for the knotted case.

In many cases, when we delete the vertices of degree $2$, the graph will become a classical graph such as complete graph, complete bipartite graph, 1-skeletons of regular polyhedra and so on. The word ``essentially'' will mean that the vertices of degree $2$ are deleted.

The first two figures give the general case, the two infinite series.

$\Gamma_{15E}^{a}(g)$ is  the complete bipartite graph $K_{(2,g+1)}$.
$\Gamma_{19}^{a}(k^2)$ is essentially the complete bipartite graph $K_{(k+1,k+1)}$. The following figures are for $g=8$ and $k=4$.

\begin{center}
\centerline{\scalebox{0.3}{\includegraphics{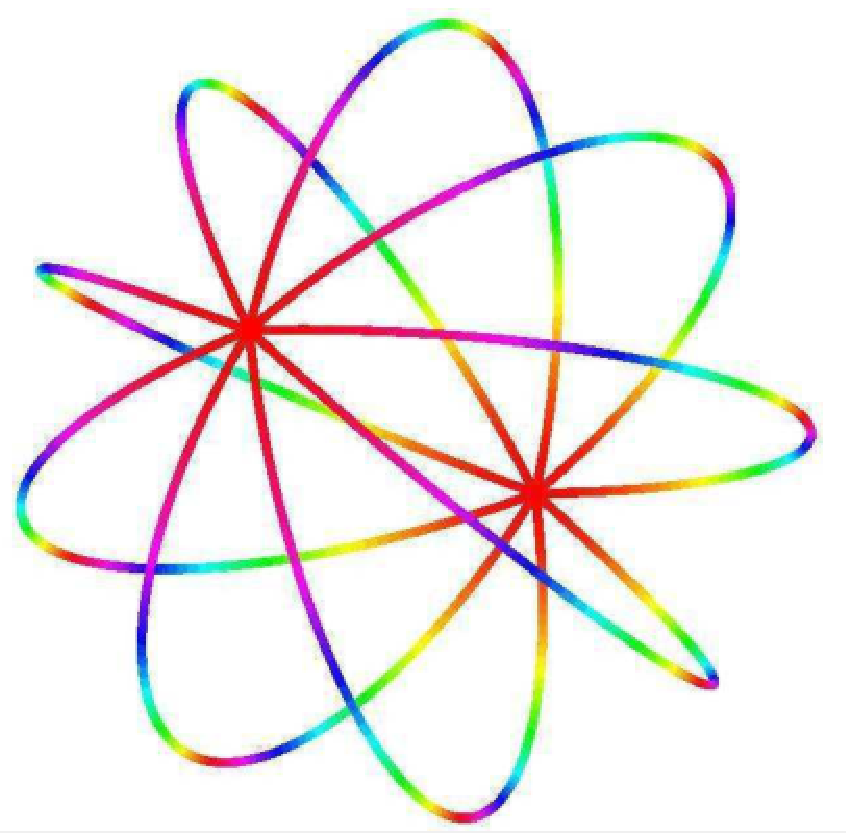}}
\scalebox{0.32}{\includegraphics{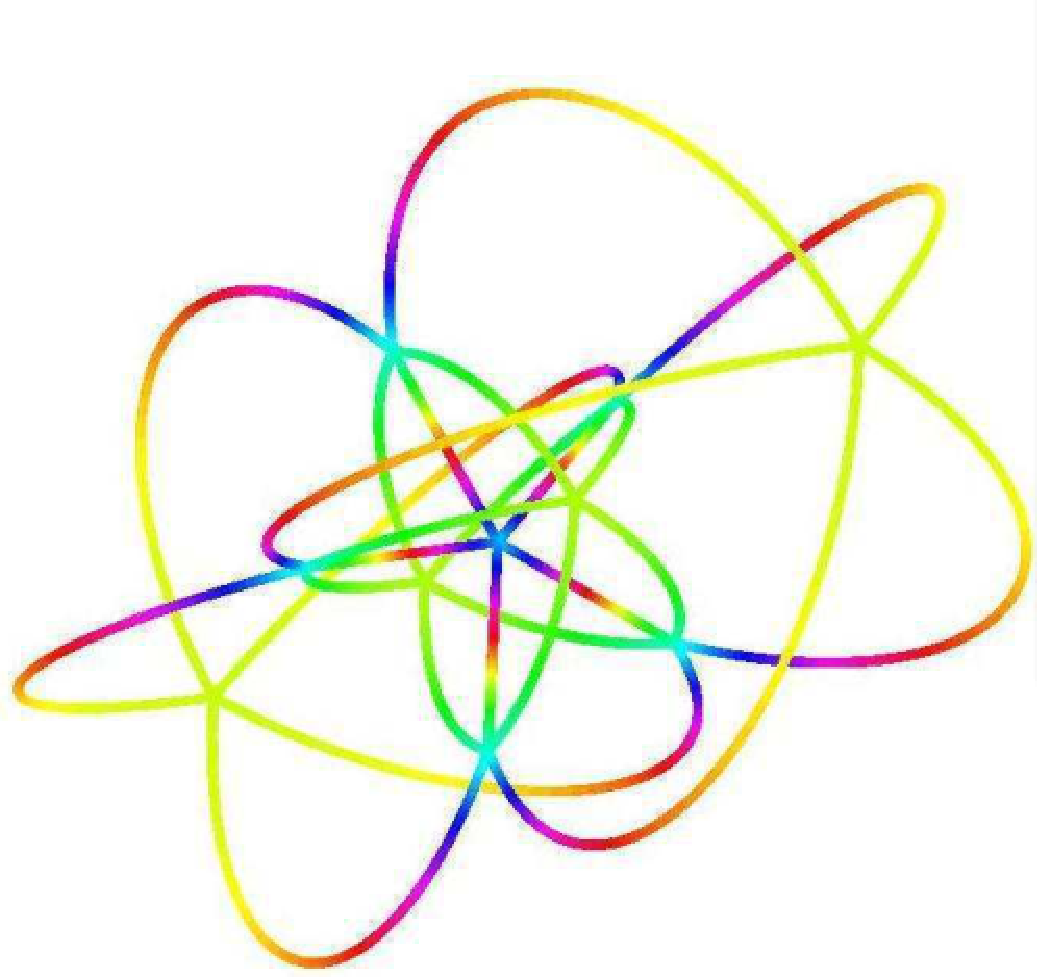}}}
\end{center}


Next we present figures related to regular 3-dimensional polyhedra.

$\Gamma_{26}^{a'}(3)$,  $\Gamma_{27}^{a'}(5)$, $\Gamma_{28}^{a'}(11)$,  $\Gamma_{27}^{b'}(7)$ and  $\Gamma_{28}^{b'}(19)$ are essentially the 1-skeletons of
the regular tetrahedron,   the regular cube, the  regular dodecahedron,  the regular octahedron and  the  regular icosahedron respectively.
 If we connect the vertices and the face centers of the  regular dodecahedron (or the regular icosahedron) by geodesic segments in the faces, then we will get the graph $\Gamma_{28}^{c'}(29)$. Note that the dual graphs of the above six graphs are all dipole graphs with vertices at the center of the polyhedra and the infinity. The following figures are for $\Gamma_{26}^{a'}(3)$, $\Gamma_{27}^{a'}(5)$ and
 $\Gamma_{28}^{a'}(11)$ respectively.

\begin{center}
\centerline{
\scalebox{0.2}{\includegraphics{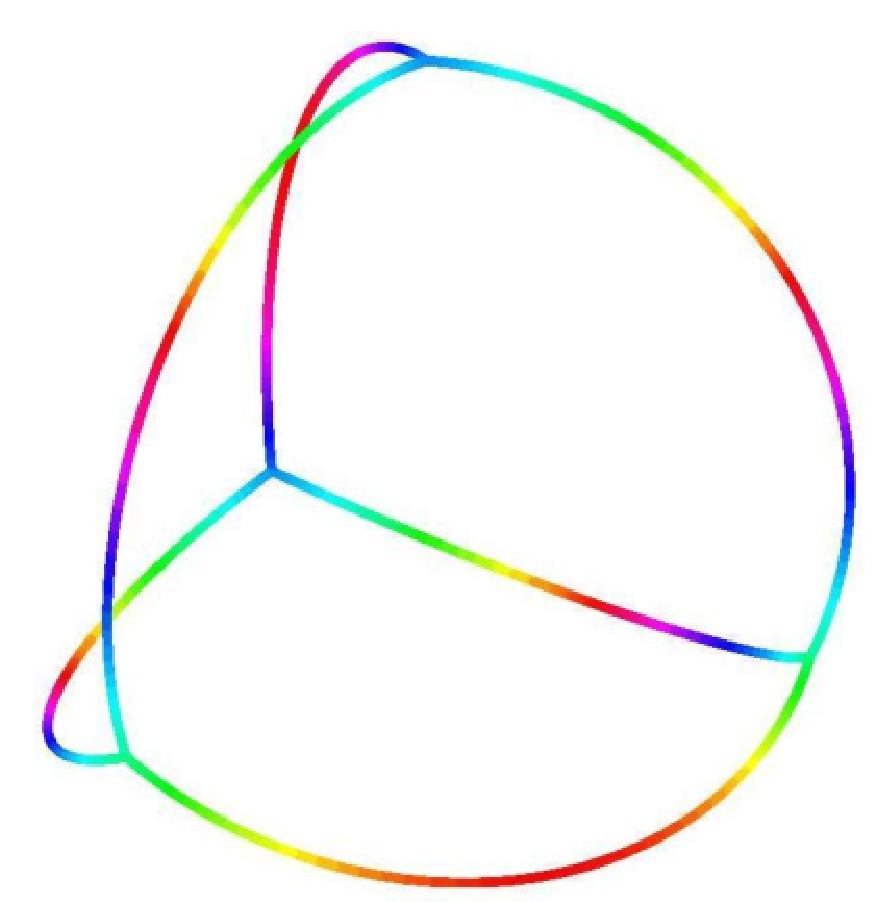}}
\scalebox{0.15}{\includegraphics{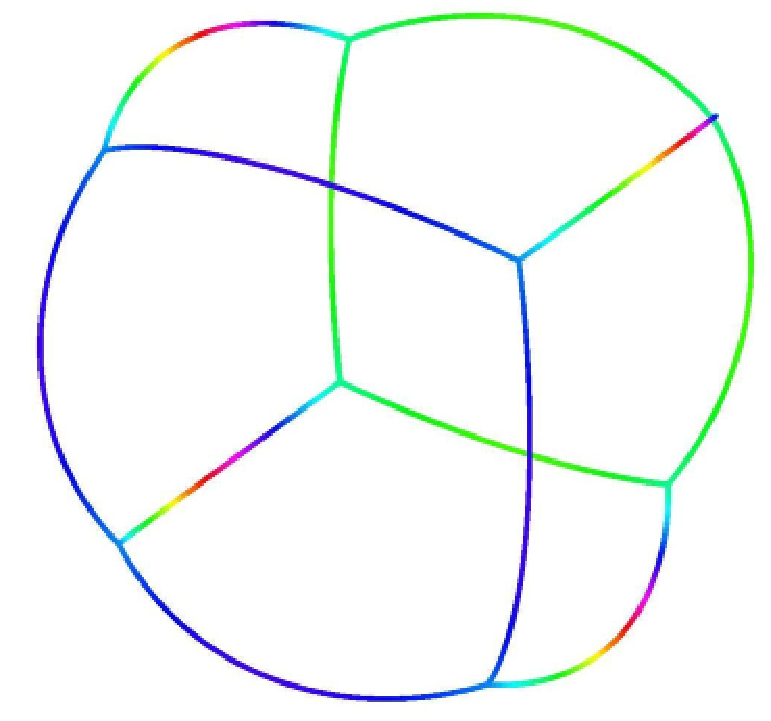}}
\scalebox{0.2}{\includegraphics{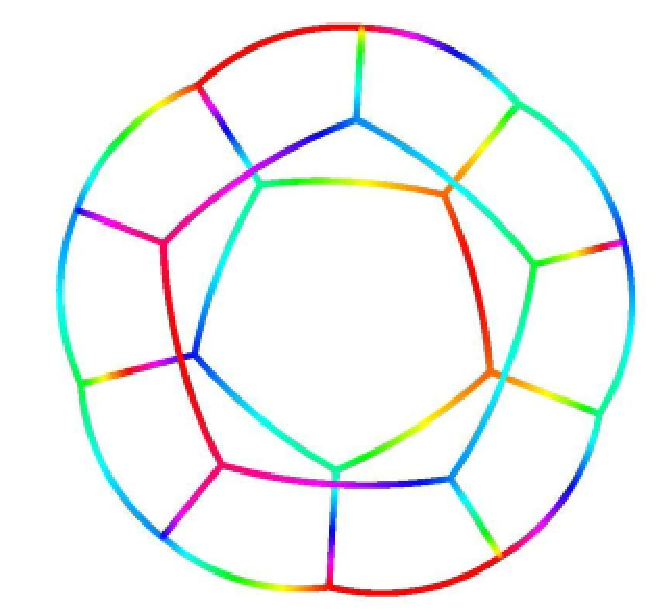}}}
\end{center}

Consider the Hopf fibration $S^3\rightarrow S^2$.
Fixed a point in $S^3$, the above six graphs in $S^2$ can be lifted piecewise, and we can get six graphs in $S^3$,
which are essentially $\Gamma_{20C}^{a}(9)$, $\Gamma_{21A}^{a}(25)$, $\Gamma_{22B}^{a}(121)$, $\Gamma_{21B}^{a}(49)$, $\Gamma_{22C}^{a}(361)$ and $\Gamma_{22D}^{a}(841)$, corresponding to $\Gamma_{26}^{a'}(3)$, $\Gamma_{27}^{a'}(5)$, $\Gamma_{28}^{a'}(11)$, $\Gamma_{27}^{b'}(7)$, $\Gamma_{28}^{b'}(19)$ and $\Gamma_{28}^{c'}(29)$. The following figures  are for $\Gamma_{20C}^{a}(9)$ and $\Gamma_{21A}^{a}(25)$ respectively.

\begin{center}
\centerline{\scalebox{0.9}{\includegraphics{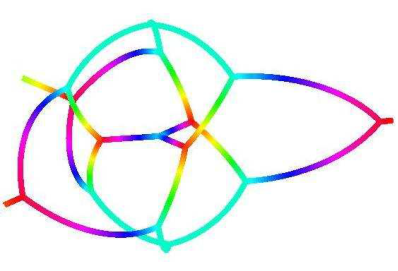}}
\scalebox{0.55}{\includegraphics{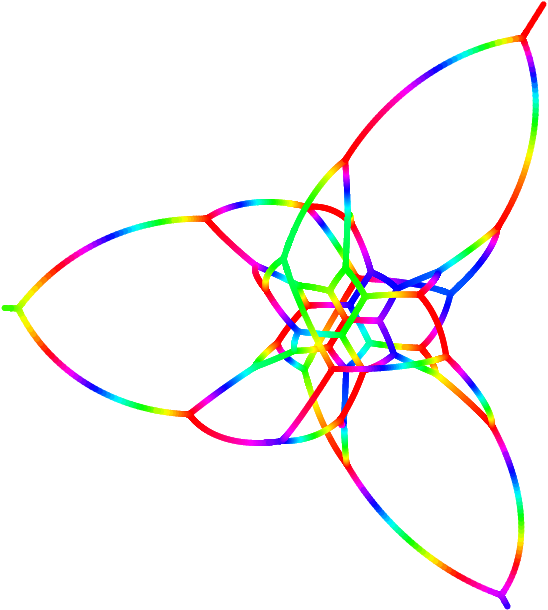}}}
\end{center}



The following graphs are related to 4-dimensional regular polyhedra.

$\Gamma_{24}^{a}(6)$ is essentially the 1-skeleton of a regular 4-simplex (or a regular 5-cell). 
If we connect the vertices and the 3-dimensional face centers of the 4-simplex by line segments in the faces, then we will get $\Gamma_{34}^{a}(11)$.
The dual graph $\Gamma_{34}^{a'}(11)$ can be obtained by connecting the 1-dimensional face centers and the 2-dimensional face centers of the 4-simplex. The following figures are for $\Gamma_{24}^{a}(6)$, $\Gamma_{34}^{a}(11)$ and $\Gamma_{34}^{a'}(11)$
respectively.


\begin{center}
\centerline{ \scalebox{0.33}{\includegraphics{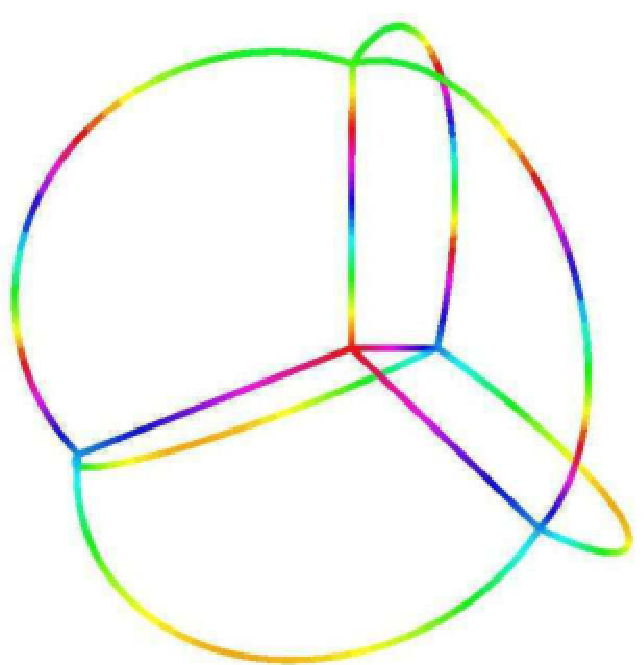}}
\scalebox{0.33}{\includegraphics{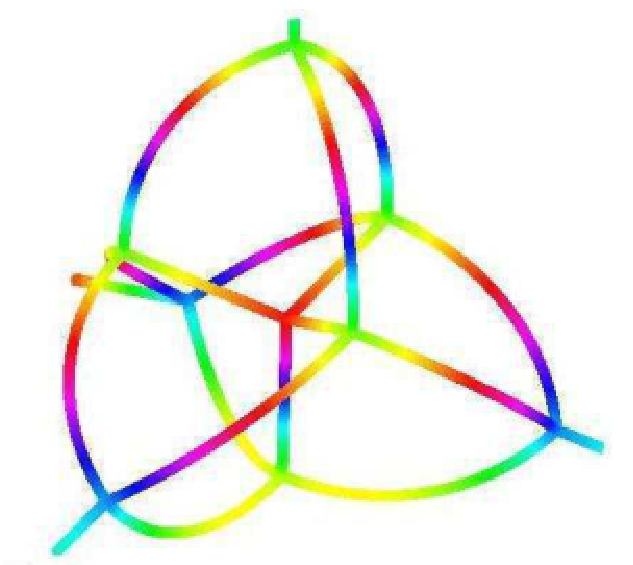}}
\scalebox{0.31}{\includegraphics{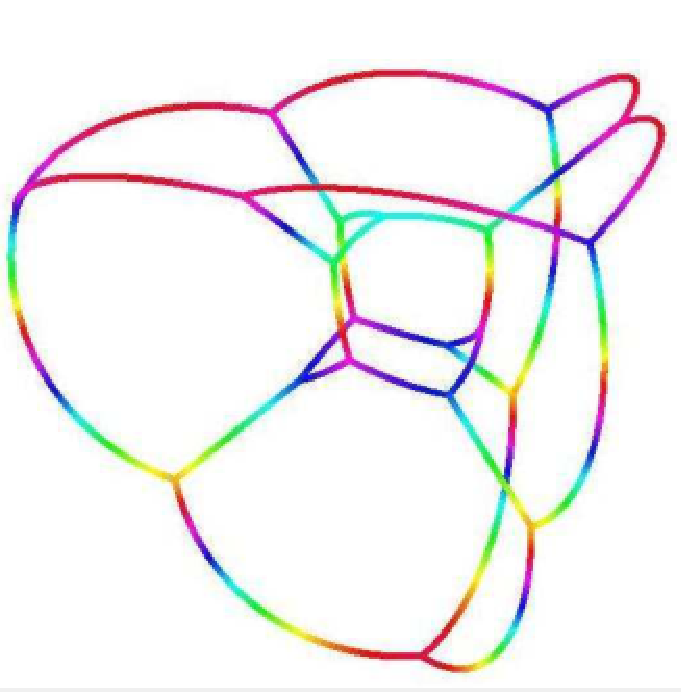}}}
\end{center}

$\Gamma_{29}^{a'}(17)$ is essentially the 1-skeleton of a regular 4-cube (or a regular 8-cell). Its dual graph $\Gamma_{29}^{a}(17)$ is essentially the 1-skeleton of a regular 16-cell which is the dual polyhedron of the 4-cube. The following figures are for $\Gamma_{29}^{a'}(17)$ and $\Gamma_{29}^{a}(17)$ respectively.

\begin{center}
\centerline{ \scalebox{0.3}{\includegraphics{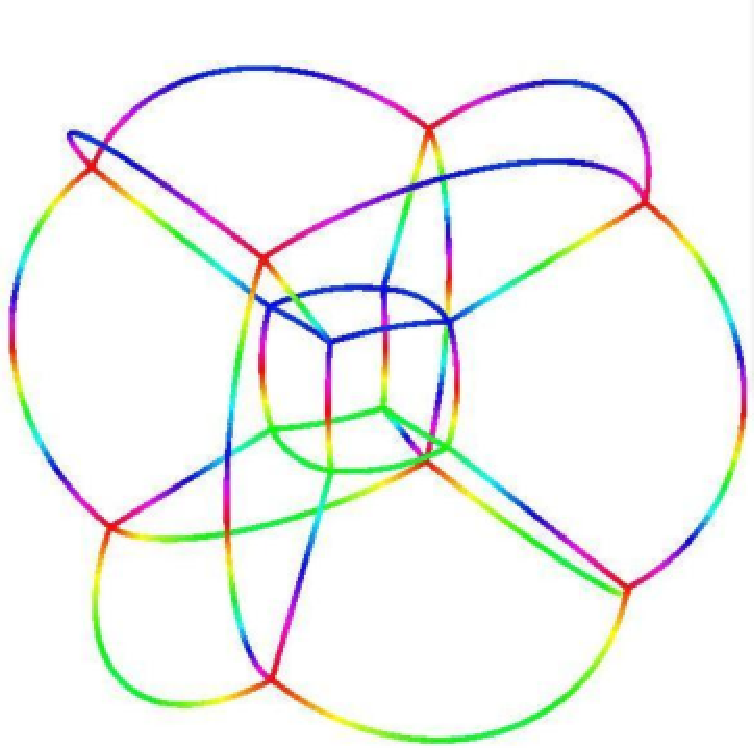}}
\scalebox{0.8}{\includegraphics{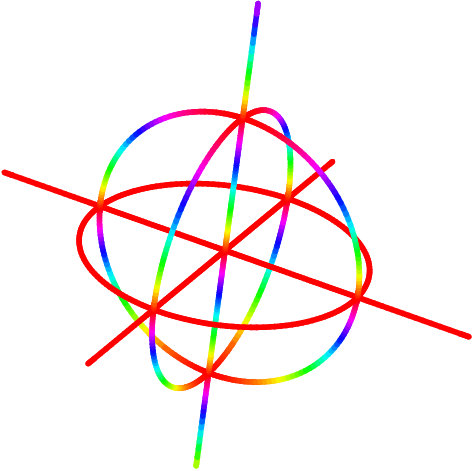}}}
\end{center}

For other regular polyhedra, there are similar constructions.
$\Gamma_{30}^{a'}(601)$ is essentially the 1-skeleton of a regular 120-cell. Its dual graph $\Gamma_{30}^{a}(601)$ is essentially the 1-skeleton of a regular 600-cell which is the dual polyhedron of the 120-cell.
$\Gamma_{25}^{a}(73)$ is essentially the 1-skeleton of a regular 24-cell, which is a self-dual polyhedron. If we connect the vertices and the 3-dimensional face centers of the 4-cube (resp. 120-cell), then we will get the graph $\Gamma_{29}^{b'}(41)$ (resp. $\Gamma_{30}^{b'}(1681)$). If we connect the 1-dimensional face centers and the 2-dimensional face centers of the 4-cube (resp. 120-cell), then we will get the graph $\Gamma_{29}^{b}(41)$ (resp. $\Gamma_{30}^{b}(1681)$).

The remaining spatial MS graph of genus $11$ is $\Gamma_{34}^{b}(11,k)$. It comes from a certain edge in the fundamental domain of the corresponding group action, see Example 7.4 of \cite{WWZZ}. The following figure is for $\Gamma_{34}^{b}(11,k)$.

\begin{center}
\centerline{ \scalebox{0.22}{\includegraphics{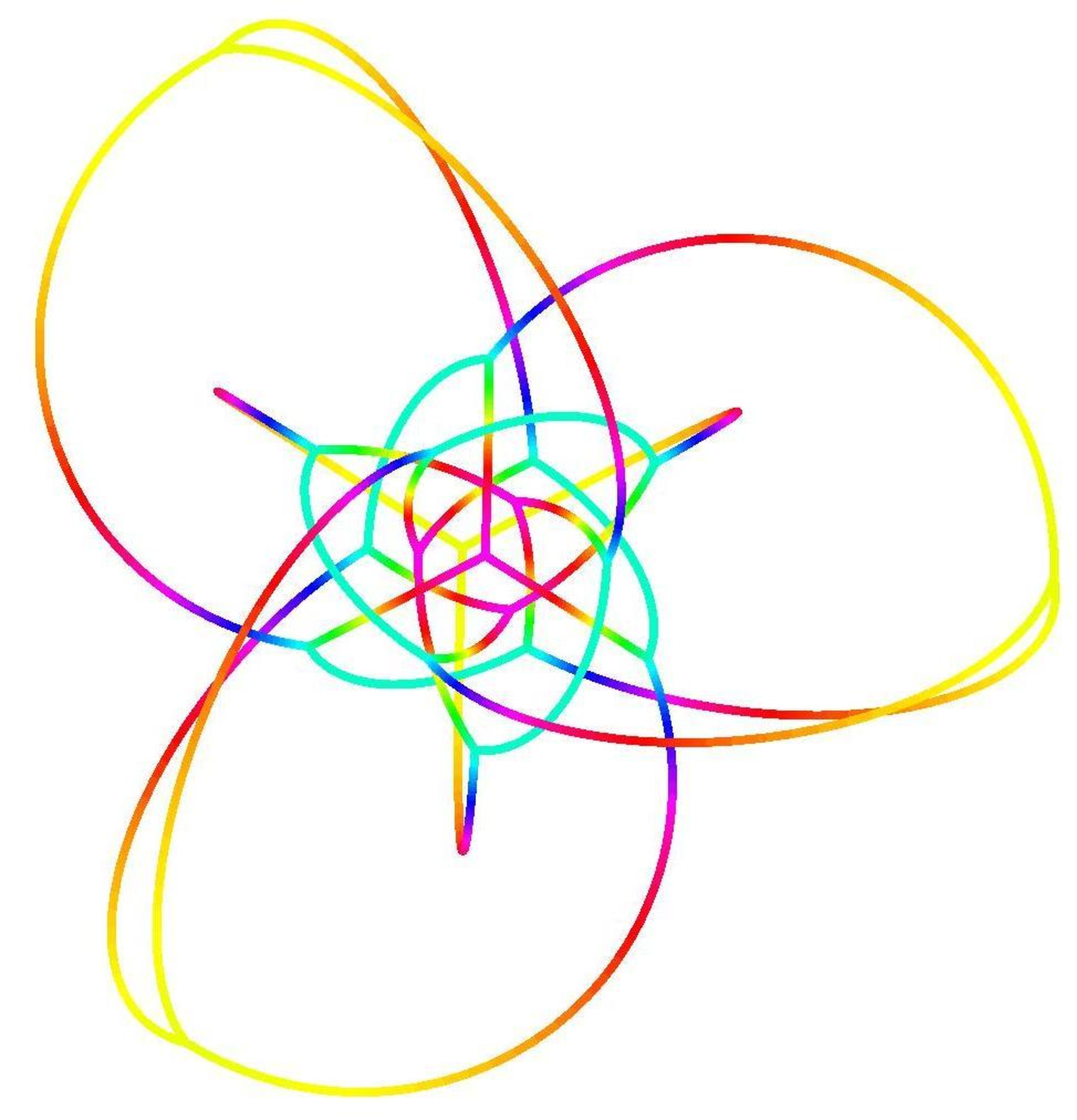}}}
\end{center}



\begin{remark} Contrary to the case of abstract MS graph, there is no uniform way to picture all the spatial MS graphs. In each case, one has to analyse the isometries of the groups so geometry is involved. To get a picture, we usually first picture a part of the graph, then apply the group action
to it to get the whole graph. The quaternion representation of the group is the key for many cases. For more details about finite subgroups of $SO(4)$ one can see \cite{CS}, \cite{Du1} and \cite{Du2}. The code for \cite{Mathematica} of the graphs can be found in Appendix C of \cite[pages 112-138]{Wa}. Even as a Ph.D thesis in Peking University, \cite{Wa} is written in Chinese, but the code part  could still be understood.
\end{remark}

\end{appendix}

\noindent Chao Wang, School of Mathematical Sciences,
University of Science and Technology of China, Hefei, 230026 CHINA\\
{\it E-mail address:} chao\_{}wang\_{}1987@126.com

\noindent Shicheng Wang, School of Mathematical  Sciences,
Beijing 100871 CHINA\\
{\it E-mail address:} wangsc@math.pku.edu.cn

\noindent Yimu Zhang, Mathematics School, Jilin University,
Changchun, 130012 CHINA\\
{\it E-mail address:} zym534685421@126.com

\noindent Bruno Zimmermann, Universita degli Studi di Trieste,
Trieste, 34100 Trieste ITALY\\
{\it E-mail address:} zimmer@units.it
\end{document}